\documentclass[10pt]{amsart}
\usepackage[dvips]{graphicx}
\usepackage{epsf}

\newcommand{\cyr}{\renewcommand\rmdefault{wncyr}%
                  \renewcommand\sfdefault{wncyr}%
                  \renewcommand\encodingdefault{OT2}%
                  \normalfont\selectfont}
\DeclareTextFontCommand{\textcyr}{\cyr}

\newtheorem{theorem}{Theorem}[section]
\newtheorem{lm}[theorem]{Lemma}

\newtheorem{prop}[theorem]{Proposition}
\theoremstyle{definition}

\theoremstyle{remark}

\def\theequation{\arabic{section}.\arabic{equation}}
\def\thetheorem{\arabic{section}.\arabic{theorem}}

\def\De{\Delta}

\def\al{\alpha}
\def\be{\beta}
\def\ga{\gamma}
\def\la{\lambda}
\def\ka{\kappa}
\def\om{\omega}
\def\pa{\partial}
\def\fhi{\varphi}
\def\eps{\varepsilon}
\def\de{\delta}

\def\tr{\mbox{\rm tr }}
\def\na{\nabla}

\def\sgn{\mbox{\rm sgn\,}}

\def\ints{\int\limits}

\def\sums{\sum\limits}
\def\ds{\displaystyle}
\def\ovr{\overline}
\def\dv{\mbox{\rm div}}
\def\lang{\langle}
\def\rang{\rangle}
\def\re{\mbox{Re}}
\def\im{\mbox{Im}}
\def\nr{\Vert}

\def\mute{\cdot}
\def\hh{\widehat{h}}
\newcommand{\e}{\` e}
\newcommand{\ust}{{ustep}}
\newcommand{\vst}{{vstep}}
\newcommand{\id}{\mbox{Id}}
\newcommand{\arcsinh}{\mbox{arcsinh}}

\begin{document}                                                 
    \title[]{On complex-valued 2D eikonals. \\
Part four: continuation past a caustic}                                 
    \author[R. Magnanini]{Rolando Magnanini}                                
    \address{Dipartimento di Matematica ``U.~Dini'',
Universit\` a di Firenze, viale Morgagni 67/A, 50134 Firenze, Italy.}                                    
    \email{magnanin@math.unifi.it}                                      
    \urladdr{http://www.math.unifi.it/users/magnanin}                                    
    \thanks{}   
\author[G. Talenti]{Giorgio Talenti}
    \address{Dipartimento di Matematica ``U.~Dini'',
Universit\` a di Firenze, viale Morgagni 67/A, 50134 Firenze, Italy.}           \email{talenti@math.unifi.it}    
\begin{abstract}
Theories of monochromatic high-frequency electromagnetic fields have been designed by Felsen, Kravtsov, Ludwig 
and others with a view to portraying features that are ignored by geometrical optics. 
These theories have recourse to eikonals that encode information on both phase and amplitude --- in other words, are complex-valued. 
The following mathematical principle is ultimately behind the scenes: 
any geometric optical eikonal, which conventional rays engender in some light region, 
can be consistently continued in the shadow region beyond the relevant caustic, provided an alternative eikonal, 
endowed with a non-zero imaginary part, comes on stage. 
\par
In the present paper we explore such a principle in dimension $2.$ 
We investigate a partial differential system that governs the real and the imaginary parts 
of complex-valued two-dimensional eikonals, and an initial value problem germane to it. 
In physical terms, the problem in hand amounts to detecting waves that rise beside, but on the dark side of, a given caustic. 
In mathematical terms, such a problem shows two main peculiarities: on the one hand, 
degeneracy near the initial curve; on the other hand, ill-posedness in the sense of Hadamard. 
We benefit from using a number of technical devices: hodograph transforms, artificial viscosity, and a suitable discretization. 
Approximate differentiation and a parody of the quasi-reversibility method are also involved. 
We offer an algorithm that restrains instability and produces effective approximate solutions.
\end{abstract}                
    
\maketitle  

\section{Introduction}

\subsection{}
Geometrical optics fits well a variety of issues, but especially survives as
an asymptotic theory 
of monochromatic high-fre\-quen\-cy electromagnetic fields ---
\cite{BB}, \cite{Dui}, \cite{FKN}, \cite{GS}, \cite{Jo}, \cite{Ke}, \cite{KL}, \cite{Kli-1} and \cite{Kli-2}, \cite{KO-1}, \cite{Lun}, \cite{MO}, \cite{Ra} 
are selected apropos references. 
Generalizations have been worked out by Felsen, Kravtsov, Ludwig and their followers --- 
see e.g. \cite{CF-1} and \cite{CF-2}, \cite{EF}, \cite{Fe-1} and \cite{Fe-2}, \cite{HF}, 
\cite{K-1},\cite{K-2}, \cite{K-3} and \cite{KFA}, \cite{LBL}, \cite{Lud-1} and \cite{Lud-2}, or consult \cite{BM}, \cite{CLOT}, \cite{KO-2}. 
One is enough for successfully modeling basic optical processes, 
such as the propagation of light and the development of caustics. 
The others embrace geometrical optics and are additionally apt to account for certain optical phenomena 
--- for instance, the rise of evanescent waves past a caustic --- that are 
beyond the reach of geometrical optics. 
A leitmotif of these is allowing a keynote parameter to adjust itself to  
a standard equation, and simultaneously take complex values.
\par
The following partial differential equation
\begin{equation}
\label{eik}
\left(\frac{\pa w}{\pa x}\right)^2+\left(\frac{\pa w}{\pa y}\right)^2=n^2(x,y)
\end{equation}
underlies the mentioned theories in case the spacial dimension is $2.$ Here $x$ and $y$ denote {\it rectangular coordinates} in the Euclidean plane; 
$n$ is a {\it real-valued, strictly positive} function of $x$ and $y;$ $w$ is allowed to take both {\it real} and {\it complex} values. 
Function $n$ represents the {\it refractive index} of an appropriate (isotropic, non-conducting) two-dimensional medium ---
its reciprocal stands for velocity of propagation. Function $w$ is named {\it eikonal} according to usage, and relates 
to the asymptotic behavior of an electromagnetic field as the wave number grows large --- the real part of $w$ accounts for oscillations, 
the imaginary part of $w$ accounts for damping. 
Geometrical optics deals exclusively with real-valued eikonals, 
complex-valued eikonals distinguish a more advanced context.
Throughout the present paper we assume the refractive index is conveniently smooth, 
and consider sufficiently smooth eikonals.
\par
The following partial differential system
\begin{equation}
\label{sys}
\begin{array}{cc}
&u_x^2+u_y^2-v_x^2-v_y^2=n^2(x,y)\\
\\
&u_x v_x +u_y v_y=0
\end{array}
\end{equation}
governs those complex-valued solutions to \eqref{eik} that obey
\begin{equation*}
\label{defw}
u=\re\, w, \ \ v=\im\, w.
\end{equation*}
Observe the architecture of \eqref{sys}: 
gradients are involved through their orthogonal invariants  
--- lengths and inner product --- only. Observe also the following properties, which result from a standard test and easy algebraic manipulations. 
First, system \eqref{sys} is {\it elliptic-parabolic} or {\it degenerate elliptic.} Second, a solution array $[u\  v]$ to \eqref{sys} 
is {\it elliptic} if and only if its latter entry $v$ is free from critical points.
\par
The {\it B\"acklund transformation,} which relates $u$ and $v$ thus
\begin{equation}
\label{back}
\begin{array}{cc}
&\ds\na v= f
\left[
\begin{array}{cc}
0 &-1\\
1 & 0\end{array}
\right]\,\na u, \\ 
&\ds f^2=1-\frac{n^2}{|\na u|^2}, \ \ \sgn f=\sgn(u_x v_y-u_y v_x),
\end{array}
\end{equation}
and implies both
$$
|\na u|\ge n
$$
and
\begin{equation}
\label{eqdivu}
\dv\left\{\sqrt{1-\frac{n^2}{|\na u|^2}}\,\,\na u\right\}=0,
\end{equation}
is another, decoupled form of \eqref{sys}. 
\par
System \eqref{sys} discloses two scenarios --- the former is tantamount 
to conventional geometrical optics, the latter opens up new vistas. 
Either the following equations
$$
u_x^2+u_y^2=n^2 \ \mbox{ and } \ v_x=v_y=0
$$
hold, or the following inequalities
$$
|\na u|>n \ \mbox{ and } \ |\na v|>0,
$$
and the following equations prevail.
\begin{eqnarray}
\label{invback}
\na u= \frac1{f}
\left[
\begin{array}{cc}
0 &1\\
-1 & 0\end{array}
\right]\,\na v, \quad \frac1{f^2}=1+\frac{n^2}{|\na v|^2}, 
\end{eqnarray}
\begin{equation}
\label{eqdivv}
\dv\left\{\sqrt{1+\frac{n^2}{|\na v|^2}}\,\,\na v\right\}=0,
\end{equation}
\begin{equation}
\label{eqnondiv}
(|\na u|^4-n^2 u_y^2) u_{xx}+2 n^2 u_x u_y u_{xy}+(|\na u|^4-n^2 u_x^2) u_{yy}=
n |\na u|^2 \lang\na n, \na u\rang,
\end{equation}
\begin{equation}
\label{eqnondivv}
(|\na v|^4+n^2 v_y^2) v_{xx}-2 n^2 v_x v_y v_{xy}+(|\na v|^4+n^2 v_x^2) v_{yy}+
n |\na v|^2 \lang\na n, \na v\rang=0.
\end{equation}
\par
Equations \eqref{invback} invert the B\"acklund transformation 
mentioned above, and imply \eqref{eqdivv}. \eqref{eqnondiv} and \eqref{eqnondivv} are quasi-linear partial differential equations of the second order in non-divergence form.
The former has a {\it mixed elliptic-hyperbolic} character: 
a solution $u$ is {\it elliptic} or {\it hyperbolic} depending on whether the length of $\na u$ exceeds, 
or is smaller than $n.$ The latter is {\it elliptic-parabolic} or {\it degenerate elliptic:} 
a solution $v$ such that $\na v$ is free from zeros is {\it elliptic,} 
degeneracy occurs at the {\it critical points} of $v.$
\par
Any real-valued sufficiently smooth solution to \eqref{eqdivu} satisfies \eqref{eqnondiv}. 
A real-valued function is an elliptic solution to \eqref{eqnondiv} if and only if 
it coincides with the former entry of an elliptic solution to \eqref{sys}. 
A real-valued function $u,$ smooth and without critical points, 
is a hyperbolic solution to \eqref{eqnondiv} if and only if 
two real-valued smooth functions $\fhi$ and $\psi$ exist such that
$$
\fhi_x^2+\fhi_y^2=n^2, \ \ \psi_x^2+\psi_y^2=n^2, \ \ \fhi_x \psi_y -\fhi_y \psi_x\not=0,
$$
and
$$
u=\frac12 (\fhi+\psi).
$$
\par
Any real-valued, sufficiently smooth solution to \eqref{eqdivv} satisfies \eqref{eqnondivv}. Any {\it elliptic} solution to \eqref{eqnondivv} satisfies \eqref{eqdivv}. A solution to \eqref{eqnondivv} 
{\it need not} satisfy \eqref{eqdivv}: for instance, perfectly smooth solutions 
to \eqref{eqnondivv} exist whose gradient vanishes exclusively in a set of measure $0,$ and which make the left-hand side of \eqref{eqdivv} a non-zero distribution.
\par
Let $J$ be endowed with an appropriate domain and obey
$$
J(v)=\iint j\left(\frac{|\na v|}{n}\right) n^2\, dx dy
$$
for any $v$ from that domain --- here $j$ is the arc length along a parabola, videlicet
$$
j(\rho)=\frac{\rho}{2}\sqrt{1+\rho^2}+\frac12\log(\rho+\sqrt{1+\rho^2})
$$
for any real $\rho.$ 
The following properties hold. (i) $J$ is convex, coercive and sub-differentiable, 
but not Fr\'echet-differentiable. (ii) Any critical point of $J,$ i.e. any function $v$ such that
$$
\pa J(v)\ni 0,
$$
satisfies \eqref{eqdivv} in any open set $\mathcal O$ such that
$$
\mathcal O \mbox{ is essentially contained in }
\{ (x,y)\in \mbox{domain of } v : \na v(x,y)\not= 0\}.
$$
Consequently, a critical point of $J$ solves a {\it free-boundary problem} 
for equation \eqref{eqdivv} --- the relevant free boundary is
$$
(\mbox{domain of } v) \cap \pa \{ (x,y)\in \mbox{domain of } v : \na v(x,y)\not= 0\}
$$
and plays the role of a {\it caustic.} (ii) Any function $v$ such that
$$
J(v)= \mbox{minimum}
$$
satisfies \eqref{eqnondivv} in an appropriate {\it viscosity sense.} 
In other words, a minimizer of $J$ solves in a generalized sense 
a {\it boundary value problem} for equation \eqref{eqnondivv}.
\par
An early treatment of \eqref{sys} traces back to \cite{ER}. 
Further apropos information is offered in \cite{MaTa1}, \cite{MaTa2}, \cite{MaTa3}, and \cite{MaTa4}, 
where solutions in closed form, qualitative features, exterior boundary value problems, 
related free boundaries, variational and viscosity methods are discussed.

\subsection{}
Geometrical optics ultimately amounts to manipulating: 
(i) the Riemannian metric known as {\it travel time,} videlicet
$$
n(x,y)\sqrt{dx^2 + dy^2};
$$
(ii) the members of appropriate one-parameter families of travel time geodesics ---
called {\it rays;} (iii) the envelopes of rays --- called {\it caustics.}
\par
Geometric optical eikonals are entirely controlled by rays. 
They shine in light regions (those spanned by relevant rays), burn out beside caustics 
(where the ray system breaks down), 
and shut down in shadow regions (that rays avoid).
As a consequence, geometrical optics is unable 
to account for any optical process that takes place beyond a caustic, on the dark side of it. 
Essentials of two-dimensional geometrical optics, which are instrumental here, are outlined in Appendix A.
\par
Complex-valued eikonals are more flexible. 
As the cited work of Felsen, Krav\-tsov and Ludwig may prompt, 
complex-valued eikonals look apt to {\it consistently continue} geometric optical 
cognates {\it into shadow regions.} Such a continuation is the main theme of the present paper.
\par
Let us pave the way by heuristically considering  the case where refractive index $n$ is $1.$ 
A classical recipe informs how general geometric optical eikonals can be cooked up: 
start from a complete integral, derive a one-parameter family of solutions, 
take the relevant envelope, and shake well. Let  $f$  be an arbitrary, 
but sufficiently smooth, real function. The following pair
\begin{equation}
\label{pair}
w=x \cos t+y \sin t +f(t), \quad 0=-x \sin t +y \cos t +f'(t)
\end{equation}
causes $w$ and $t$ to enjoy the following properties: 
$$
\left(\frac{\pa w}{\pa x}\right)^2+\left(\frac{\pa w}{\pa y}\right)^2=1,
$$
the pertinent {\it eikonal equation} governing $w;$
$$
\cos t\, \frac{\pa t}{\pa x}+\sin t\, \frac{\pa t}{\pa y}=0,
$$
a {\it Burgers-type equation} governing $t;$
$$
w_x t_x+w_y t_y=0,
$$
showing that the gradients of $w$ and $t$ are {\it orthogonal;}
$$
w_x=\cos t, \quad w_y=\sin t,
$$
a {\it B\"acklund transformation} further relating $w$ and $t.$ 
Both the rays of $w$ and the level lines of $t$ are the straight-lines where
$$
-x \sin t +y \cos t +f'(t)=0
$$
and $t$ equals a constant. Such straight-lines span the {\it light region} 
and envelope the {\it caustic.} We have
$$
 -f'= \mbox{the support function of the caustic,}
$$
and 
$$
x=f''(t) \cos t+f'(t) \sin t, \ \ y=f''(t) \sin t-f'(t) \cos t, \ \ w=f(t)+f''(t)
$$
along the caustic. Therefore the second-order derivatives of $w$ and the gradient of $t$ 
simultaneously blow up there --- in particular, the caustic of $w$ is also the shock-line of $t.$
\par
We claim that both $w$ and $t$ can be {\it continued beyond the caustic,} 
in a subset of the shadow region, if {\it complex values} are allowed. 
Suppose $f$ is analytic --- so as it can be continued by a holomorphic function of a complex variable. Let
$i=\sqrt{-1},$ the unit imaginary number. Let $u, v, \la, \mu$ be real; put the following equations
$$
w=u+iv, \ \ t=\la+i\mu,
$$
and equations \eqref{pair} together, but force \eqref{pair} to produce real $x$ and $y.$ The following formulas
$$
x=\frac{\sin\la}{\cosh\mu}\,\re f'(t)+\frac{\cos\la}{\sinh\mu}\,\im f'(t), \ \
y=-\frac{\cos\la}{\cosh\mu}\,\re f'(t)+\frac{\sin\la}{\sinh\mu}\,\im f'(t),
$$
$$
u=\re f'(t)+(x \cos\la+y \sin\la)  \cosh\mu, \ v=\im f'(t)+(-x \sin\la+y \cos\la)  \sinh\mu,
$$
ensue, then an inspection testifies that the claimed continuation ensues too.
\par
Incidentally, we have also shown that solutions to the non-viscous Burgers equation 
can be continued past the shock-line by suitable complex-valued 
solutions to the same equation. Let us mention that complex-valued solutions 
to viscous and non-viscous Burgers equation 
are dealt with in \cite{KWYZ}, \cite{PS} and \cite{KeOk}.
A more exhaustive analysis is carried out in Appendix B.

\subsection{}
A certain initial value problem for system \eqref{sys} --- 
the one described in items (i) and (ii) below, and in figure 1 --- summarizes our purposes.

\begin{figure}[tbh]
\label{fig:caustica}
\centerline{
{\epsfxsize=3.5in
\epsfbox{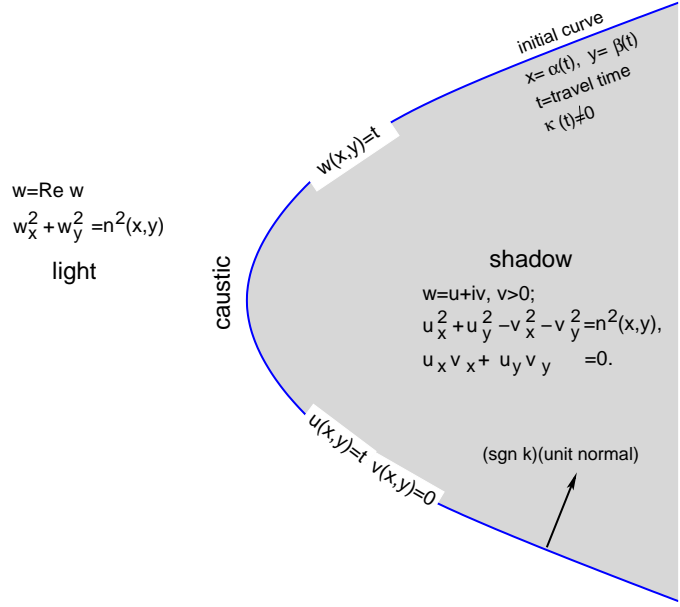}
} 
}
\caption{A geometric optical eikonal and its
continuation past a caustic.}
\end{figure}

\par
(i) An initial curve IC is given. The following alternative applies: 
either IC is specified {\it exactly} --- no error infects the definition of IC; 
or else IC is a {\it phantom} --- some coarse and polluted sampling of IC is gotten.
 In the former case assume IC is smooth enough. 
In the latter case {\it recover} IC, i.e. fed the available data 
into an appropriate denoising process, and then elect the consequent output 
as an operative substitute of IC --- ad hoc tools are provided in Appendix C.
\par
Represent IC (either the authentic one, or else its surrogate) by the following equations
\begin{equation}
\label{curve}
x=\al(t), \quad y=\be(t), 
\end{equation}
and adjust parameter $t$ so as
\begin{eqnarray*}
\label{ttime}
t=\mbox{\it a travel time}
\end{eqnarray*}
without any loss of generality.
\par
Assume travel time is an extra metric in action and the relevant {\it geodesic curvature} of IC is 
{\it free from zeroes.} In other words, postulate that \eqref{curve} and either of the following equations
\begin{equation*}
\label{geodesic1a}
\ka\, (\mbox{velocity})^3=\mbox{Geodesic curvature},
\end{equation*}
\begin{equation*}
\label{geodesic1}
\ka\, (\mbox{velocity})^2=\mbox{Euclidean curvature}-\langle \mbox{unit normal}, \na\log n(x,y) \rangle
\end{equation*}
result in
\begin{equation*}
\label{denonvanish}
\ka \mbox{ \it vanishes nowhere}.
\end{equation*}
(ii) A pair $[u\, v]$ is sought that obeys system \eqref{sys} and fulfills the following conditions. First,
\begin{equation}
\label{initial}
u(x,y)=t, \quad v(x,y)=0,
\end{equation}
if $x,y$ and $t$ are subject to \eqref{curve}.
Second, $u$ and $v$ are defined in the side of IC that
$$
(\sgn\, \ka)\times(\mbox{unit normal to IC})
$$
points to.
\par
Arguments from Appendix A allow us to comment as follows. 
IC and the mentioned side of it can be viewed as a {\it caustic} and a {\it shadow region,} respectively. 
Any geometric optical eikonal, which makes IC a caustic, lives in the illuminated side of IC; 
the complex-valued eikonal, whose real part is $u$ and whose imaginary part is $v,$ 
lives in the opposite, dark side of IC; both equal a travel time along IC. 
An {\it extension} of the geometric optical eikonal in hand ensues. 
Such an extension does obey the eikonal equation, 
lives in both the light region and a subset of the shadow region, 
and takes complex values where shadow prevails. 
In physical terms, problem (i) and (ii) accepts a caustic in input, then models 
{\it evanescent waves} that rise in the dark side of it.
\par
In the present paper we focus our attention on solutions $[u\, v]$ to the problem (i) and (ii) that meet the following extra requirements:
\par\noindent
(iii) they are {\it elliptic;}
\par\noindent
(iv) their Jacobian determinant obeys
$$
\sgn(u_x v_y-u_y v_x)=\sgn\ka.
$$
Condition (iii) ensures that the latter entry $v$ is {\it not constant;}
as will emerge from subsequent developments, initial conditions 
plus conditions (iii) and (iv) ensure that the same entry is {\it nonnegative.}
\par
The solutions to problem (i)-(iv) develop {\it singularities} near IC, as any 
geometric optical eikonal does in the vicinity of the relevant caustic. 
We will show that they obey the following expansions
\begin{equation}
\label{shadow}
u(x,y)=s+o(r), \
v(x,y)=\frac{2\sqrt{2}}{3}|\ka(s)|^{1/2}|r|^{3/2}+o(r^{3/2})
\end{equation}
as $(x,y)$ belongs to the appropriate side of, and is close enough to IC. Here $r$ and $s$ 
stand for the curvilinear coordinates described in Appendix A --- 
informally, $r$ is a signed distance from IC, $s$ is a lifting of a travel time inherent to IC.
Note the physical meaning of \eqref{shadow}: the damping effects, which are encoded in
the imaginary part of a complex-valued eikonal, are tuned by the geodesic curvature 
of the relevant caustic.

\setcounter{equation}{0}
\setcounter{theorem}{0}

\section{Framework}

\subsection{}
A convenient coordinate system must be called for. We choose to recast \eqref{sys} 
by reversing the roles of dependent and independent variables --- i.e. we think
of $u$ and $v$ as curvilinear coordinates, and think of $x$ and $y$ as functions of $u$ and $v.$ 
In other words, we subject \eqref{sys} to the change of variables that is called {\it hodograph 
transformation} at times --- see \cite{Eva} and \cite{Z}, for instance. 
\par
Let $[u\, v]$ be any smooth elliptic solution to \eqref{sys}, and observe the following.
\par\noindent
(i) The level lines of $u$ and those of $v$ are free from singular points, 
and cross at a right angle. Moreover, the Jacobian determinant 
\begin{equation}
\label{jac}
u_x v_y-u_y v_x \mbox{ vanishes nowhere.} 
\end{equation}
\par
Let
\begin{eqnarray*}
\label{mappings}
[u\,v]\mapsto [x(u,v)\,y(u,v)] \mbox{ be a local {\it inverse} of } 
[x\,y]\mapsto [u(x,y)\,v(x,y)],
\end{eqnarray*}
then observe the following equation
\begin{equation*}
\label{jacjac}
(u_x v_y-u_y v_x)(x_u y_v-x_v y_u)=1
\end{equation*}
and the propositions (ii)-(iv) below.
\par\noindent
(ii) The following partial differential system holds
\begin{equation}
\label{firstform}
1/E-1/G=1, \ \ F=0.
\end{equation}
Here
$$
E=n^2(x,y) (x_u^2+y_u^2), \ F=n^2(x,y) (x_u x_v+y_u y_v), \ G=n^2(x,y) (x_v^2+y_v^2)
$$
--- in other words,
$$
n^2(x,y)((dx)^2+(dy)^2)=E (du)^2+2 F du\,dv+G (dv)^2.
$$
\par\noindent
(iii) The following systems and equations hold
\begin{equation}
\label{matrixsys}
\frac{\pa}{\pa u}\left[\begin{array}{cc}
x\\
y
\end{array}\right]=
f\,\left[\begin{array}{cc}
0 &1\\
-1 &0
\end{array}\right]
\frac{\pa}{\pa v}\left[\begin{array}{cc}
x\\
y
\end{array}\right],
\end{equation}
\begin{equation}
\label{matrixsys2}
\frac{\pa}{\pa v}\left[\begin{array}{cc}
x\\
y
\end{array}\right]=
\frac1{f}\,\left[\begin{array}{cc}
0 &-1\\
1 &0
\end{array}\right]
\frac{\pa}{\pa u}\left[\begin{array}{cc}
x\\
y
\end{array}\right],
\end{equation}
\begin{equation}
\label{deff}
\begin{array}{c}
\ds\sgn f=\sgn(x_u y_v-x_v y_u), \\
\ds\frac1{f^2}=1+n^2(x,y)(x_v^2+y_v^2), \ \
f^2=1-n^2(x,y)(x_u^2+y_u^2).
\end{array}
\end{equation}
\par\noindent
(iv) The following equations hold
\begin{equation}
\label{deff2}
f^2=\frac{1-n^2(x,y)x_u^2}{1+n^2(x,y)x_v^2}, \ f^2=\frac{1-n^2(x,y)y_u^2}{1+n^2(x,y)y_v^2},
\end{equation}
\begin{equation}
\label{matrixsys3}
\frac{\pa^2}{\pa u^2}\left[\begin{array}{cc}
x\\
y
\end{array}\right]+f^4 \frac{\pa^2}{\pa v^2}\left[\begin{array}{cc}
x\\
y
\end{array}\right]=
(1-f^2)\,\left[\begin{array}{cc}
x_u^2-y_u^2 &-2 x_u y_u\\
2 x_u y_u &x_u^2-y_u^2
\end{array}\right]
\na\log n(x,y)
\end{equation}
--- in the event that $n$ is identically $1,$ these equations read
\begin{equation}
\label{matrixsys3n1}
x_{uu}+\left(\frac{1-x_u^2}{1+x_v^2}\right)^2 x_{vv}=0, \ \ y_{uu}+\left(\frac{1-y_u^2}{1+y_v^2}\right)^2 y_{vv}=0.
\end{equation}

\begin{proofA}
System \eqref{sys} tells us that
$$
|\na u|>0,
$$
and that the gradients of $u$ and $v$ are orthogonal. An assumption gives
$$
|\na v|>0.
$$
The following equation
$$
(u_x v_y-u_y v_x)^2=|\na u|^2 |\na v|^2-\lang\na u, \na v\rang^2
$$
concludes the proof.
\end{proofA}

\begin{proofB}
Since
\begin{eqnarray*}
&\ds\frac{\pa(u,v)}{\pa(x,y)}\times\frac{\pa(x,y)}{\pa(u,v)}=\left[\begin{array}{cc}
1 &0\\
0 &1
\end{array}\right],\\
&\ds\frac{\pa(u,v)}{\pa(x,y)}\times\left[\frac{\pa(u,v)}{\pa(x,y)}\right]^T=
\left[
\begin{array}{cc}
|\na u|^2 &\lang\na u, \na v\rang\\
\lang\na u, \na v\rang &|\na v|^2
\end{array}
\right],\\
&\ds\left[\frac{\pa(x,y)}{\pa(u,v)}\right]^T\times\frac{\pa(x,y)}{\pa(u,v)}=
\left[
\begin{array}{cc}
x_u^2+y_u^2 &x_u x_v+y_u y_v\\
x_u x_v+y_u y_v &x_v^2+y_v^2
\end{array}
\right],
\end{eqnarray*}
we have
\begin{eqnarray*}
&\ds\frac1{x_u^2+y_u^2}=|\na u|^2 - \frac{\lang\na u, \na v\rang^2}{|\na v|^2},\
\frac1{x_v^2+y_v^2}=|\na v|^2 - \frac{\lang\na u, \na v\rang^2}{|\na u|^2},\\
&\ds x_u x_v+y_u y_v=-\lang\na u, \na v\rang (x_u y_v-x_v y_u)^2.
\end{eqnarray*}
\par
The conclusion ensues.
\end{proofB}

\begin{proofC}
The latter equation from \eqref{firstform} yields
$$
x_u:y_v=y_u:(-x_v),
$$
hence the following systems result
$$
\left[\begin{array}{cc}
x_u\\
y_u
\end{array}\right]=f\,
\left[\begin{array}{cc}
0 &1\\
-1 &0
\end{array}\right]
\left[\begin{array}{cc}
x_v\\
y_v
\end{array}\right], \ \
\left[\begin{array}{cc}
x_v\\
y_v
\end{array}\right]=\frac1{f}\,
\left[\begin{array}{cc}
0 &-1\\
1 &0
\end{array}\right]
\left[\begin{array}{cc}
x_u\\
y_u
\end{array}\right].
$$
\par
Factor $f$ is easily identified. Both the above systems give
$$
\sgn f=\sgn(x_u y_v-x_v y_u);
$$
the former and system \eqref{firstform} imply
$$
\frac1{f^2}=1+n^2(x,y)(x_v^2+y_v^2);
$$
the latter and system \eqref{firstform} imply
$$
f^2=1-n^2(x,y)(x_u^2+y_u^2).
$$
\par
Another arrangement reads
$$
\left[\begin{array}{cc}
x_u\\
x_v
\end{array}\right]=
\left[\begin{array}{cc}
0 &f\\
-1/f &0
\end{array}\right]
\left[\begin{array}{cc}
y_u\\
y_v
\end{array}\right], \ \
\left[\begin{array}{cc}
y_u\\
y_v
\end{array}\right]=
\left[\begin{array}{cc}
0 &-f\\
1/f &0
\end{array}\right]
\left[\begin{array}{cc}
x_u\\
x_v
\end{array}\right]
$$
--- two B\"acklund transformations, inverse of one another. The former and system \eqref{firstform} imply
$$
f^2=\frac{1-n^2(x,y)y_u^2}{1+n^2(x,y)y_v^2},
$$
the latter and system \eqref{firstform} imply
$$
f^2=\frac{1-n^2(x,y)x_u^2}{1+n^2(x,y)x_v^2}.
$$
\par
The integrability conditions, which pertain to the B\"acklund transformations in hand, read
$$
\left[\frac{\pa}{\pa u}\ \frac{\pa}{\pa v}\right]
\left[\begin{array}{cc}
1/f &0\\
0 &f
\end{array}\right]
\left[\begin{array}{cc}
x_u &y_u\\
x_v &y_v
\end{array}\right]=0
$$
and result in equation \eqref{matrixsys3} after algebraic manipulations.
\par
Equations \eqref{deff} and \eqref{deff2} are consistent with 
one another and with the early equations \eqref{back} and \eqref{invback}, as Proposition (ii) and its proof show.
\par
The proof is complete.
\end{proofC}

\subsection{}
In view of the discussion above, problem (i)-(iv) stated in 
Subsection 1.3 amounts to looking for solutions $[x\, y]$ 
to the following partial differential system
\begin{equation}
\label{pdesys}
\begin{array}{c}
\ds\frac{\pa}{\pa v}
\left[\begin{array}{cc}
x\\
y
\end{array}\right]=\frac1{f}\,
\left[\begin{array}{cc}
0 &-1\\
1 & 0
\end{array}\right]\,
\frac{\pa}{\pa u}
\left[\begin{array}{cc}
x\\
y
\end{array}\right], \\
\\
\ds\sgn f=\sgn\ka, \ \ f^2=1-n^2(x,y)(x_u^2+y_u^2),
\end{array}
\end{equation}
that are defined either in the half-strip
$$
a<u<b, \ 0<v<\infty
$$
or in an appropriate bounded subset if it, and satisfy the following initial conditions
\begin{equation}
\label{sysinitial}
x(u,0)=\al(u), \quad y(u,0)=\be(u) \ \mbox{ for } \ a\le u\le b.
\end{equation}
\par
We will concentrate on such a problem. A behavior of solutions 
$[x\, y]$ to \eqref{pdesys} and \eqref{sysinitial} as $v$ is close to $0$ 
is fixed up in the next section. An algorithm for computing 
the same solutions is offered in Section 4. Section 5 is devoted to an example.

\setcounter{equation}{0}
\setcounter{theorem}{0}

\section{Behavior near the caustic}

The state of affairs causes any solution of \eqref{pdesys} and \eqref{sysinitial} 
to suffer from {\it singularities} near the initial line --- indeed,
\begin{equation*}
\label{infty}
x_v^2(u,v)+y_v^2(u,v)\to \infty
\end{equation*}
as $a\le u\le b$ and $v$ approaches $0.$ The proposition below offers more 
details on the subject, as well as a proof of expansions \eqref{shadow}.

\begin{prop} 
Let $x$ and $y$ obey system \eqref{pdesys} and initial conditions \eqref{sysinitial}. 
Assume $x(u,v)$ and $y(u,v)$ depend smoothly upon $u$ for every nonnegative, sufficiently small $v.$ Then
the following asymptotic expansion holds
\begin{equation*}
\label{expansion}
\left[\begin{array}{cc}
x(u,v)\\
y(u,v)\end{array}\right]=
\left[\begin{array}{cc}
\al(u)\\
\be(u)\end{array}\right]+
\frac{(3v)^{2/3} \sgn\,\ka(u)}{2|\ka(u)|^{1/3}\sqrt{\al'(u)^2+\be'(u)^2}}
\left[\begin{array}{cc}
-\be'(u)\\
\al'(u)\end{array}\right]+o(v^{2/3})
\end{equation*}
as $a\le u\le b$ and $v$ is positive and approaches $0.$
\end{prop}
\begin{proof}
A hypothesis made on $\al$ and $\be$ in Subsection 1.3, equations \eqref{pdesys} 
and initial conditions \eqref{sysinitial} tell us that
$$
f(u,v)\to 0
$$
as $a\le u\le b$ and $v$ is positive and approaches $0.$ Therefore,
$$
\lim_{v\downarrow 0}\frac{v^{1/3}}{f(u,v)}=
\lim_{v\downarrow 0} \sgn f(u,v)\,\left\{\sgn f(u,v)\left/\frac{\pa}{\pa u} f^3(u,v)\right.\right\}^{1/3}, 
$$
by L'Hospital's rule.
Equations \eqref{pdesys} give successively
\begin{eqnarray*}
&\ds\frac{\pa}{\pa v} n^2(x,y)=\frac{2\, n^2(x,y)}{f(u,v)} \left\lang \na\log n(x,y),\left[\begin{array}{cc}
-y_u\\
x_u
\end{array}\right]\right\rang,\\
&\ds\frac{\pa}{\pa v} (x_u^2+y_u^2)=-\frac{2}{f(u,v)}(x_u y_{uu}-x_{uu} y_u),
\end{eqnarray*}
and
\begin{multline*}
\frac{\pa}{\pa v} f^3(u,v)=3\,(x_u^2+y_u^2)^{3/2}n^2(x,y)
\left\{\frac{x_u y_{uu}-x_{uu} y_u}{(x_u^2+y_u^2)^{3/2}}-\right.\\
\left.\left\lang \na\log n(x,y),(x_u^2+y_u^2)^{-1/2}\left[\begin{array}{cc}
-y_u\\
x_u
\end{array}\right]\right\rang\right\}.
\end{multline*}
We infer 
$$
\lim_{v\downarrow 0}\, (3v)^{1/3}\left[\begin{array}{cc}
x_v\\
y_v
\end{array}\right]=\frac{\sgn\,\ka(u)}{|\ka(u)|^{1/3}\sqrt{\al'(u)^2+\be'(u)^2}}\left[\begin{array}{cc}
-\be'(u)\\
\al'(u)
\end{array}\right]
$$
because of the very definition of $\ka.$
\par
The conclusion follows.
\end{proof}

\setcounter{equation}{0}
\setcounter{theorem}{0}
\section{Discrete setting}

\subsection{}
Rendering problem \eqref{pdesys} and \eqref{sysinitial} into an effective 
discrete form entails coping with {\it singularities} of solutions, {\it overflows,} 
and {\it ill-posedness} in the sense of Hadamard.
\par
Singularities result from features of both the system and 
the initial conditions in hand, as already remarked in the previous section. 
Overflows take place whenever the constraint
$$
n^2(x,y)(x_u^2+y_u^2)<1
$$
chances to be violated. Note that the system
$$
\frac{\pa}{\pa v}
\left[
\begin{array}{c}
x\\ y
\end{array}
\right]=
\left[
\begin{array}{cc}
0 &1\\
-1 &0
\end{array}
\right]
\frac{\pa}{\pa u}
\left[
\begin{array}{c}
x\\ y
\end{array}
\right]
$$
--- Cauchy-Riemann, a possible linearized version of \eqref{pdesys} --- possesses obvious solutions
$$
\left[
\begin{array}{c}
x\\ y
\end{array}
\right]=\exp(-\sqrt{t}+t v)
\left[
\begin{array}{c}
\cos(t u)\\ \sin(t u)
\end{array}
\right]
$$
which are highly instable, i.e.
$$
\sup\left\{
\left[\frac{\pa^j x}{\pa u^j}(u,0)\right]^2+\left[\frac{\pa^j y}{\pa u^j}(u,0)\right]^2:
-\infty<u<\infty
\right\}\to 0
$$
as $t\uparrow\infty$ and $j=1, 2,\dots,$ and
$$
\inf\left\{
x^2(u,v)+y^2(u,v): -\infty<u<\infty
\right\}\to \infty
$$
as $t\uparrow\infty$ and $v$ is positive.
\par
Ill-posedness is a typical drawback of initial value problems for partial differential equations 
and systems of elliptic type. It was observed by Hadamard \cite{Had1}-\cite{Had2}, 
and deeply investigated by John \cite{Jo1}-\cite{Jo3}, Lavrentiev jr. \cite{La1}-\cite{La3}, \cite{LV}, 
Miller \cite{Mlr1}-\cite{Mlr3}, Payne \cite{Pay1}-\cite{Pay4} and \cite{PaSa1}-\cite{PaSa2}, 
Pucci \cite{Pc1}-\cite{Pc4}, and Tikhonov \cite{Tkh1}-\cite{Tkh5}. 
Classical surveys on the subject have been authored by Lavrentiev jr., \cite{La4}, 
\cite{LRS}, Payne \cite{Pay5}, and Tikhonov \cite{TA}. 
Related information is in \cite{Brt}-\cite{Brt1}, \cite{BDV}, \cite{BG}, 
\cite{Isa}, \cite{IVT}, \cite{Mor}, \cite{Nsh}, \cite{Rch} and \cite{Tl}. 
A sample of more recent contributions includes \cite{Ali}, \cite{AI}, \cite{AK},
\cite{Bll}, \cite{Brt2}, \cite{BKP}, \cite{BE}, \cite{BR}, \cite{BV2}, \cite{Bre}, 
\cite{Cnn}, \cite{Crr}, \cite{CD}, \cite{CE}, \cite{Co}, \cite{Cdv}, \cite{De},
\cite{El1}-\cite{El2}, \cite{Ew1}-\cite{Ew2}, \cite{HH}, \cite{HR}, \cite{Hf}, 
\cite{KV}, \cite{Kno}, \cite{LV}, \cite{Lns}, \cite{Msl1}-\cite{Msl2}, \cite{Mnk}, 
\cite{Ntr1}-\cite{Ntr2}, \cite{Pay6}-\cite{Pay8}, \cite{Pmh}, \cite{RE}, 
\cite{RHH}, \cite{RS}, \cite{Ro}, \cite{Stn}, \cite{TGSY}, \cite{TLY}, \cite{VFC}, \cite{W}.
\par
As experience suggests, one might attempt to contend with ill-posedness 
via {\it a priori} bounds on solutions and similar devices. 
We opt not to touch on this issue in the present paper, 
and focus our attention on constructive aspects instead.
\par
We rely upon: (i) {\it asymptotic expansions,} describing how relevant solutions 
behave near the initial line; (ii) a technique of {\it approximate differentiation,} 
especially designed for working in presence of errors; 
(iii) an appropriate injection of {\it artificial viscosity,} 
which softens a coefficient and protects against overflows; 
(iv) an imitation of the {\it quasi-reversibility method.}

\subsection{}
Besides data exposed to view, our method takes six parameters in input: 
$M, N, \la, \mu, \nu, \xi$. The first two are large integers 
that specify the number of samples in hand --- e.g. $M = N = 100.$ 
The remaining four parameters set the tone of smoothing processes:
$\la$ and $\mu$ relate to approximate differentiation; $\nu$ stands for viscosity; 
$\xi$ relates to quasi-reversibility.

\subsection{}
The following equations and inequality
\begin{eqnarray*}
&\ds \ust=(b-a)/(M-1), \ \ \vst>0,\\
&\ds u_j=a+(j-1)\,\ust \ \ (j=1,\dots, M), \ \ 
v_k=(k-1)\,\vst \ \ (k=1,\dots, N), 
\end{eqnarray*}
will be in force throughout this section. We choose a mesh to consist of the following points
$$
(u_j, v_k) \ \ (j=1,\dots, M; k=1,\dots, N),
$$
and store sample values at mesh points in the following matrices
$$
[x(j,k)]_{j=1,\dots, M; k=1,\dots, N}, \ \ [y(j,k)]_{j=1,\dots, M; k=1,\dots, N}.
$$
\par
The columns of these matrices, i.e.
$$
x(\cdot,1), x(\cdot,2),\dots, x(\cdot,N), \ \  y(\cdot,1), y(\cdot,2),\dots, y(\cdot,N),
$$
are recursively generated in the following way.
\par\noindent
{\bf First step.} 
$$
x(j,1)=\al(u_j), \ y(j,1)=\be(u_j) \ \ (j=1,\dots, M),
$$
according to initial conditions \eqref{sysinitial}.
\par\noindent
{\bf Second step.} 
\begin{eqnarray*}
&\ds x(j,2)=x(j,1)+\frac{\sgn\ka(u_j)}{2 |\ka(u_j)|^{1/3}}
\frac{-\be'(u_j)}{\sqrt{\al'(u_j)^2+\be'(u_j)^2}}\,(3 v_2)^{2/3},\\
&\ds y(j,2)=y(j,1)+\frac{\sgn\ka(u_j)}{2 |\ka(u_j)|^{1/3}}
\frac{\al'(u_j)}{\sqrt{\al'(u_j)^2+\be'(u_j)^2}}\,(3 v_2)^{2/3} \ \ (j=1,\dots, M),
\end{eqnarray*}
according to expansion \eqref{expansion}.
\par\noindent
{\bf Further steps.} For  $k=3,\dots, N$ do actions (i)-(iii) below.
\par\noindent
(i) Differentiate.
\begin{eqnarray*}
&\ds X=x(\cdot, k-1), \ Y=y(\cdot, k-1),\\
&\ds A=D X, \ B=D Y.
\end{eqnarray*}
Here 
\begin{eqnarray*}
&\ds \la=\mbox{a dimensionless positive parameter},\\
&\ds \frac{\mu}{b-a}\simeq 0.0415+0.5416\times M^{-1/2}-0.6426\times M^{-1}+1.3706\times M^{-3/2},\\
&\ds K(u)=\frac1{\pi}\int_0^\infty \frac{\cos(ut)}{1+t^8}\, dt
\end{eqnarray*}
for every $u,$ and
$$
D=\frac1{\mu}\left[K'\left(\frac{u_i-u_j}{\mu}\right)\right]_{i,j=1,\dots, M}
\left\{\la\, \id+\left[K\left(\frac{u_i-u_j}{\mu}\right)\right]_{i,j=1,\dots, M}\right\}^{-1}
$$
--- a matrix that mimics differentiation with respect to variable $u,$ 
and is analyzed in Appendix C. 
\par\noindent
(ii) Enter viscosity.
$$
f(j)=P(\nu,n(X(j),Y(j))\sqrt{A(j)^2+B(j)^2} \, \sgn\ka(u_j)  \ \ (j=1,\dots,M),
$$
$$
U=-\left[
\begin{array}{cccccc}
f(1) &0 &\dots &0\\
0 &f(2) &\dots &0\\
\vdots &\vdots &\ddots &\vdots\\
0 &0 &\dots &f(M)
\end{array}
\right]^{-1}\!\! B,
\ \
V=\left[
\begin{array}{cccccc}
f(1) &0 &\dots &0\\
0 &f(2) &\dots &0\\
\vdots &\vdots &\ddots &\vdots\\
0 &0 &\dots &f(M)
\end{array}
\right]^{-1}\!\! A.
$$
\vskip.1cm\noindent
Here
$$
0<\nu=\mbox{artificial viscosity}\le\pi/2,
$$
and $P$ is given by
$$
P(\nu,\rho)=\left\{
\begin{array}{ll}
\ds\sqrt{1-\rho^2} \ &\mbox{ if } \ 0\le\rho\le (1+\sin^2 \nu)^{-1/2},\\
\ds\frac{\sin \nu}{\rho+\sqrt{\rho^2-\cos^2 \nu}}\ &\mbox{ if } \ \rho>(1+\sin^2 \nu)^{-1/2}
\end{array}\right.
$$
--- observe that
\begin{eqnarray*}
& 0\le \rho\mapsto P(\nu,\rho) \ \mbox{is strictly positive and continuously differentiable},\\
&\ds P(\nu,\rho)\ \mbox{approaches } \sqrt{1-\rho^2} \
\mbox{uniformly as $0\le\rho\le 1$ and $\nu$ approaches $0.$}
\end{eqnarray*}
\par\noindent
(iii) Enter quasi-reversibility.
$$
x(\cdot,k)=\fhi(v_{k}), \ \ y(\cdot,k)=\psi(v_{k}).
$$
Here $\fhi$ and $\psi$ are the vector-valued mappings that are generated thus:
$$
\xi=\mbox{a dimensionless positive parameter};
$$
$$
R=(\ust)^{-2}\left[
\begin{array}{cccccc}
2 &-5 &4 &-1 &\cdots &0\\
1 &-2 &1 &0  &\cdots &0\\
\ddots &\ddots &\ddots &\ddots &\ddots &\ddots\\
0 &\cdots &0 &1 &-2 &1\\
0 &\cdots &-1 &4 &-5 &2
\end{array}
\right],
$$
a caricature of a second-order derivative; 
\begin{eqnarray*}
& \fhi(v_{k-1})=X, \ \ \fhi'(v_{k-1})=U,\\
&\ds 3 \xi\, \nr R \fhi(v_k)\nr^2+(\vst)\,
\ints_{v_{k-1}}^{v_k}\nr\fhi''(v)\nr^2 dv=\mbox{minimum},
\end{eqnarray*}
\begin{eqnarray*}
& \psi(v_{k-1})=Y, \ \ \psi'(v_{k-1})=V,\\
&\ds 3 \xi\, \nr R \psi(v_k)\nr^2+(\vst)\,
\ints_{v_{k-1}}^{v_k}\nr\psi''(v)\nr^2 dv=\mbox{minimum}.
\end{eqnarray*}
As is easy to see,
\begin{eqnarray*}
&\fhi''''(v)=0 \ \mbox{ if } \ \ v_{k-1}<v<v_k,\\
&\fhi''(v_k)=0, \ \fhi'''(v_k)=3\xi (\vst)^{-2} (R^T R)\, \fhi(v_k);\\
&\psi''''(v)=0 \ \mbox{ if } \ \ v_{k-1}<v<v_k,\\
&\psi''(v_k)=0, \ \psi'''(v_k)=3\xi (\vst)^{-2} (R^T R)\, \psi(v_k).   
\end{eqnarray*}
Therefore,
\begin{eqnarray*}
&x(\cdot,k)=\left(\id+\xi\,\vst\,(R^T R)\right)^{-1}(X+\vst\, U),\\
&\ds y(\cdot,k)=\left(\id+\xi\,\vst\,(R^T R)\right)^{-1}(Y+\vst\, V).
\end{eqnarray*}
---  in other words $x(\cdot,k)$ and $y(\cdot,k)$ are mollified versions of $X+\vst\, U$ and $Y+\vst\, V,$ respectively.
\par
{\bf Last step.} End.

\subsection{} 
As a matter of fact, the above process simulates the following partial differential system 
\begin{eqnarray*}
&\ds \frac{\pa}{\pa v}
\left[
\begin{array}{c}
x\\ y
\end{array}
\right]=
\frac1{f}\,\left[
\begin{array}{cc}
0 &-1\\ 
1 &0
\end{array}
\right]
\frac{\pa}{\pa u}
\left[
\begin{array}{c}
x\\ y
\end{array}
\right]-\xi^4
\frac{\pa^4}{\pa u^4}
\left[
\begin{array}{c}
x\\ y
\end{array}
\right],\\
\\
&\ds f=P\left(\nu, n(x,y) \sqrt{x_u^2+y_u^2}\right)\,\sgn \ka(u),
\end{eqnarray*}
where the modified, and extra, terms protect against overflows and instability. The methods based on either artificial viscosity or quasi-reversibility share basic features: they all suggest perturbing the underlying partial differential equation or system in a way or another, in order to palliate obstructions. The quasi-reversibility method was introduced in \cite{LL}, and improved in \cite{Mi}, \cite{GZ}; 
other references are \cite{AT}, \cite{ATY}, 
\cite{Bou1}-\cite{Bou2}, \cite{DR}, \cite{Ew1}, \cite{Go}, \cite{Hu}, 
\cite{HZ}, \cite{KS}, \cite{La}, \cite{Pay5}, \cite{Sho1}-\cite{Sho2}, \cite{TT}.

\setcounter{equation}{0}
\setcounter{theorem}{0}
\section{Example}

For simplicity, suppose refractive index $n$ is $1.$ 
Consider the curve --- known as {\it Tschirnhausen's cubic} or 
{\it trisectrix of Catalan} \cite{Law}, \cite{Seg} --- whose parametric equations read
$$
x=\frac12 (1-3 t^2), \ \  y=\frac{t}{2} (3-t^2), \ \ (t=\mbox{parameter}),
$$
and imply
$$
\mbox{arc length}=\frac{t}{2} (t^2+3), \ \
t=2\sinh\left(\frac13\arcsinh(\mbox{arc length})\right).
$$

\begin{figure}[tbh]
\label{fig:munch2}
\centerline{
{\epsfxsize=3.5in
\epsfbox{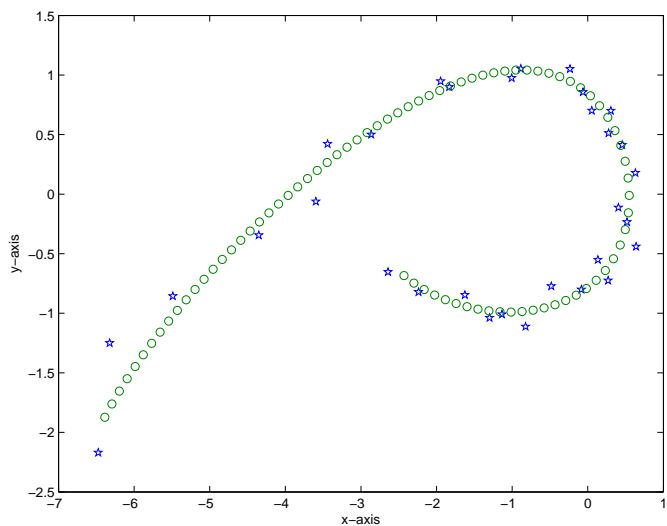}
} 
}
\caption{An arc of Tschirnhausen's cubic: 
a gross sampling (stars) and a denoised version (cicles).}
\end{figure}

\begin{figure}[tbh]
\label{fig:munch3}
\centerline{
{\epsfxsize=3.5in
\epsfbox{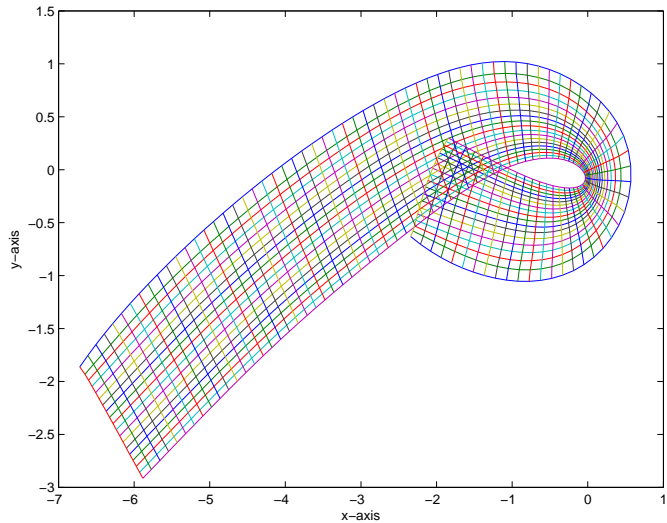}
} 
}
\caption{Level lines where either $u=$ constant or $v=$ constant.}
\end{figure}
\par
A geometric optical eikonal $w$ making Tschirnhausen's cubic 
a caustic is represented by the following equations
\begin{eqnarray*}
&\ds x=\frac12 (1-3 t^2)-\frac{2 s t}{1+t^2}, \ \ y=\frac{t}{2} (3-t^2)+\frac{s (1-t^2)}{1+t^2}, \\
&\ds w=s+\frac{t}{2} (3+t^2) \ \ (s, t=\mbox{parameters})
\end{eqnarray*}
in the light region. As arguments from Appendix B show, 
the same eikonal can be continued in the shadow region via the following equations
\begin{eqnarray*}
&\ds x=\frac{1-2 (s^2+t^2)+s^4-2 s^2 t^2-3 t^4}{2(1+s^2+t^2)}, \ \ 
y=\frac{t\,[3-2 (s^2-t^2)-(s^2+t^2)^2]}{2(1+s^2+t^2)}, \\
&\ds w=is+t-\frac{2(is+t)}{1+(is+t)^2}\, x+\frac{1-(is+t)^2}{1+(is+t)^2}\, y \ \ (s, t=\mbox{real parameters}).
\end{eqnarray*}

\begin{figure}[tbh]
\label{fig:munch1}
\centerline{
{\epsfxsize=3.5in
\epsfbox{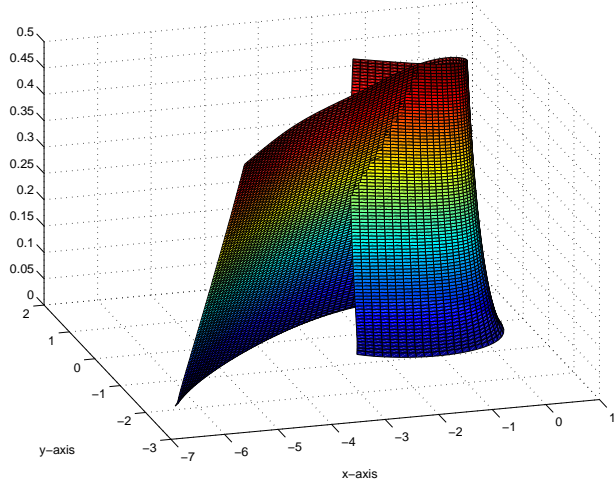}
} 
}
\caption{Plot of imaginary part $v$ versus $x$ and $y.$}
\end{figure}

\par
The method from Sections 1 to 4 goes in the following way.
\par\noindent
(i) Consider the arc of the Tschirnhausen's cubic where
$$
-1.5\le t\le 2.2,
$$
and let
$$
(\al_j, \be_j) \ \ (j=1,\dots, 31)
$$
be a gross sampling of such an arc --- in other words,
\begin{eqnarray*}
&\ds t_j=-1.5+3.7\,\frac{j-1}{30} \ \ (j=1,\dots, 31),\\
&\ds \al_j=\frac12 (1-3 t_j^2)+(5\% \mbox{ random noise}), \ \
\be_j=\frac{t_j}{2} (3-t_j^2) +(5\% \mbox{ random noise}).
\end{eqnarray*}
\par\noindent
(ii) Plug gross data into the following denoising process
\begin{eqnarray*}
&\ds \la=0.005, \ \ \mu=0.1260,\\
&\ds \sums_{j=1}^{31} \left[\al\left(\frac{j-1}{30}\right)-\al_j\right]^2+
\la \ints_{-\infty}^{\infty}[\mu^7 (\al'''')^2+\mu^{-1} \al^2]\, dt=\mbox{minimum},\\
&\ds \sums_{j=1}^{31} \left[\be\left(\frac{j-1}{30}\right)-\be_j\right]^2+
\la \ints_{-\infty}^{\infty}[\mu^7 (\be'''')^2+\mu^{-1} \be^2]\, dt=\mbox{minimum},
\end{eqnarray*}
and let the path that is represented by the following equations
$$
x=\al(t), \ \ y=\be(t), \ \ 0\le t\le 1,
$$
surrogate the original Tschirnhausen's cubic.
\par\noindent
(iii) Adjust parametric equations as follows:
\begin{eqnarray*}
&\ds\frac{dt}{du}=[\al'(t)^2+\be'(t)^2]^{-1/2}, \ t(0)=0, \\
&x=\al(t(u)), \ \ y=\be(t(u)), \ \ 0\le u\le \mbox{Length}.
\end{eqnarray*}
\par\noindent
(iv) Select requisite parameters thus
$$
M=101, \ \ N=91, \ \ \vst=0.005, \ \ \nu=0.5, \ \ \xi=0.9;
$$
and then set the algorithm from Section 4 to work.
\par
Results are shown in figures 2, 3 and 4, and comfortably agree with those drawn from closed formulas.
\vskip.1cm
\centerline{\bf Acknowledgement}
\vskip.1cm
The Authors wish to thank L. Sgheri and G.A. Viano for their valuable help and advice.

\setcounter{equation}{0}
\setcounter{theorem}{0}

\appendix 
\def\theequation{\Alph{section}.\arabic{equation}}
\def\thetheorem{\Alph{section}.\arabic{theorem}}

\section{}

Basic mathematical lineaments of two-dimensional geometrical optics are outlined in the next paragraphs. 
Selected references on geometrical optics, and on some of its generalizations and applications 
are \cite{BB}, \cite{Dui}, \cite{GS}, \cite{Jo}, \cite{Ke}, \cite{KL}, \cite{Kli-1}-\cite{Kli-2}, \cite{KK},
\cite{KO-1}, \cite{Lun}, \cite{MF}, \cite{MO}, \cite{Ra}.

\subsection{Terminology.} Let $n$ be a {\it refractive index} --- i.e. a tractable function of two real variables $x$ and $y,$ 
which takes positive values only and is bounded away from zero locally. Any real-valued, suitably smooth solution 
$w$ to \eqref{eik} is a {\it geometric optical eikonal (GOE).} The domain of $w,$ plus parts of the relevant boundary, 
is a {\it light region;} the complement of it is a {\it shadow region.} The trajectories of $\na w,$ 
namely the orbits of the following dynamical system
$$
\left| 
\begin{array}{cc}
dx &dy\\
w_x(x,y) &w_y(x,y)
\end{array}
\right|=0
$$
are called {\it lines of steepest descent} --- irrespective of whether they are genuine lines or not. 
A line of steepest descent is a {\it ray} if $w$ is twice continuously differentiable in some neighborhood of it; 
any line of steepest descent, which is not a ray, is a {\it caustic.} 
(Rays are smooth curves, which have one degree of freedom and travel all over areas without intersecting one another. 
Caustics are exceptions in a sense: loosely speaking, 
they can be thought of as envelopes of rays.) The Riemannian arc length, whose element takes the following form
$$
n(x,y)\,\sqrt{dx^2+dy^2},
$$
is known as {\it travel time} --- travel time is an alias of the customary arc length in the case where $n\equiv 1.$ 

\subsection{GOEs and travel time.}

(i) The restriction of any GOE $w$ to either an appertaining 
ray or caustic automatically coincides with a properly rescaled travel time $t.$
\par\noindent
(ii) Let {\it nodal line} be an alias of locus of zeros. 
The value of any GOE $w$ at any point $(x,y)$ equals either the 
travel time between $(x,y)$ and a nodal line of $w$ or the negative of such a travel time 
--- provided $(x,y)$ is not a long way off.
\begin{proofA} 
By definition, both rays and caustics of $w$ obey
$$
dx: w_x(x,y)= dy: w_y(x,y);
$$
the following equation
$$
t=\mbox{a rescaled travel time}
$$
is an alias of
$$
dt=\pm n(x,y) \sqrt{(dx)^2+(dy)^2}.
$$
Consequently,
$$
(dw/dt)^2=n^{-2}(x,y) (w_x^2+w_y^2)
$$
along any ray or caustic in question. Property (i) follows.
\end{proofA}
\par
Property (ii) is a consequence of (i) and Fermat's principle below. 

\subsection{Fermat's principle.} The travel time geodesics, i.e. those curves which render
$$
\int n(x,y)\,\sqrt{dx^2+dy^2} 
$$
either a minimum or stationary, are characterized by the following second-order ordinary differential equation
$$
\mbox{curvature} =\langle \mbox{unit normal},\na\log n(x,y)\rangle
$$
--- they have geodesic curvature $0,$ and are perfect straight lines in the case where $n\equiv 1.$
The rays of any GOE are geodesics with respect to travel time. The travel time geodesics that are trajectories of some proper vector field --- i.e. have one degree of freedom and are free from mutual intersections --- are the rays of some GOE.
\vskip.1cm
The foregoing statements rest upon first principles of calculus of variations, 
differential geometry and ordinary differential equations. 
Recall that the following formulas apply to any smooth parametric curve:
$$
\mbox{velocity}=\sqrt{\left(dx/dt\right)^2+\left(dy/dt\right)^2},
$$
$$
\mbox{unit tangent $=$ (velocity)}^{-1}\frac{d}{dt}
\left[\begin{array}{cc}
x\\
y
\end{array}\right], \ \ \mbox{unit normal}=\left[\begin{array}{cc}
0 &-1\\
1 &0
\end{array}\right]
(\mbox{unit tangent})
$$
$$
(\mbox{velocity})^{-1}\frac{d}{dt}(\mbox{unit tangent})=(\mbox{curvature})\times (\mbox{unit normal}),
$$
$$
(\mbox{velocity})^{-1}\frac{d}{dt}(\mbox{unit normal})=-(\mbox{curvature})\times (\mbox{unit tangent}),
$$
$$
\mbox{curvature}=(\mbox{velocity})^{-3}
\left|\begin{array}{cc}
dx/dt &d^2x/dt^2\\
dy/dt &d^2y/dt
\end{array}\right|\, ;
$$
$$
\mbox{curvature}=(\mbox{velocity})^{-3} h^{-3}\,\left| 
\begin{array}{ccc}
x(t-h) &y(t-h) &1\\
x(t) &y(t) &1\\
x(t+h) &y(t+h) &1
\end{array}
\right|+O(h^2),
$$
\begin{multline*}
\mbox{curvature}=(\mbox{velocity})^{-3} h^{-3}\,\left\{\left| 
\begin{array}{ccc}
x(t) &y(t) &1\\
x(t+h) &y(t+h) &1\\
x(t+2h) &y(t+2h) &1
\end{array}
\right| \right.\\
\left. -\frac12\,\left| 
\begin{array}{ccc}
x(t+h) &y(t+h) &1\\
x(t+2h) &y(t+2h) &1\\
x(t+3h) &y(t+3h) &1
\end{array}
\right|\right\} +O(h^2).
\end{multline*}
Recall that the following formula
\begin{eqnarray*}
&\mbox{curvature of the lines of steepest descent $=$}\\
&\mbox{Jacobian determinant of $|\na w|^{-1}$ \& $w$}
\end{eqnarray*}
applies whenever $w$ is smooth and has no critical point. Recall also that
\begin{eqnarray*}
& n(x,y)\times (\mbox{Geodesic curvature})=\\
& \mbox{Euclidean curvature}-\langle \mbox{unit normal},\na\log n(x,y)\rangle
\end{eqnarray*}
if travel time is an alternative Riemannian metric in force.

\subsection{Initial value problems and geometry of their solutions.}
The condition of taking given values along some given path is qualified {\it initial} according to usage. 
Seeking a GOE, which obeys some initial condition, is an {\it initial value problem.} 
Such a problem has either two different solutions or no solution at all, 
depending on whether the eikonal equation and the initial condition match 
or conflict. Generally speaking, a solution $w$ can be detected in the former case by successively 
detecting the objects listed below, based upon the arguments provided.
\begin{itemize}
\item
The values of $\na w$ along the initial curve.
\end{itemize}
Since the eikonal equation specifies the length of $\na w$ 
and the initial condition identifies a tangential component of $\na w$, 
the normal component of $\na w$ along the initial curve comes out in two different modes.
\begin{itemize}
\item
The rays of $w.$
\end{itemize}
An ODE reasoning demonstrates that the travel time geodesics, 
which live near the initial curve and leave it with the same direction as $\na w,$ 
are the trajectories of a smooth vector field. By Fermat's principle, 
these geodesics are the requested rays indeed.
\begin{itemize}
\item
The values of $w$ itself on each ray. 
\end{itemize}
Property (i) from Subsection 2 fits the situation well.

\subsection{Standard initial value problems.} 
The present item concerns existence, regularity and the number of GOEs that satisfy orthodox initial conditions.
\par
Assume henceforth all ingredients are smooth and let IC stand for initial curve in shorthand. Let
\begin{equation}
\label{curveapp}
x=\al(t), \quad y=\be(t), \quad a\le t\le b
\end{equation}
be a parametric representation of IC such that
\begin{equation}
\label{ttimeapp}
t=a\ travel\ time.
\end{equation}
Let the initial condition imply
\begin{equation*}
\label{icapp}
w(x,y)=\ga(t)
\end{equation*}
as $x,y$ and $t$ satisfy \eqref{curveapp}, and assume
\begin{equation}
\label{gammapp}
\left|\frac{d\ga}{dt}(t)\right|<1
\end{equation}
for $a\le t\le b.$
Then exactly two GOEs satisfy the initial condition displayed.
Moreover, these eikonals are smooth in a {\it full} neighborhood of IC --- 
the relevant light regions surround it completely.
\par
The case where refractive index $n$ is constant involves explicit formulas, 
as well as gives evidence to interactions among the eikonal equation, 
{\it Burgers-type equations} and {\it B\"acklund transformations.} Assume
$$
n\equiv 1,
$$
and let \eqref{curveapp} to \eqref{gammapp} be in force. Define $\fhi, \psi, \om, p, q$  by
\begin{eqnarray*}
&\cos\fhi=\al', \quad \sin\fhi=\be';\\
&\cos\psi=\ga', \quad \sin\psi=\pm\sqrt{1-(\ga')^2};\\
&\om=\fhi+\psi;\\
&p=\cos\om, \quad q=\sin\om.
\end{eqnarray*}
\par
Since the appurtenant Jacobian matrix equals
$$
\left[
\begin{array}{cc}
q(s) & p(s)\\
-p(s) & q(s)
\end{array}
\right]
\left[
\begin{array}{cc}
\sin\psi(s)-(w-\ga(s))\, d\om(s)/ds &\ 0\\
0 &\ 1
\end{array}
\right]
$$
the following pair
$$
\left[
\begin{array}{c}
x\\
y
\end{array}
\right]=
\left[
\begin{array}{cc}
\al(s)\\
\be(s)
\end{array}
\right]+(w-\ga(s))\,\left[
\begin{array}{cc}
p(s)\\
q(s)
\end{array}
\right]
$$
makes $s$ and $w$ implicit functions of $x$ and $y$ in a neighborhood of IC. 
The properties listed below ensue. Function $s$ obeys the following equations
\begin{eqnarray*}
&\ds -q(s)\,(x-\al(s))+p(s)\,(y-\be(s))=0,\\
&\ds p(s)\,\frac{\pa s}{\pa x}+q(s)\,\frac{\pa s}{\pa y}=0,\\
&\ds \left(\frac{\pa s}{\pa x}\right)^2+\left(\frac{\pa s}{\pa y}\right)^2=
\left[
\sin\psi(s)-(w-\ga(s))\,\frac{d\om(s)}{ds}\right]^{-2}.
\end{eqnarray*}
(The first assures that the level lines of $s$ are straight, the second is a PDE of {\it Burgers type.}) 
Function $w$ obeys the following {\it eikonal equation}
$$
\left(\frac{\pa w}{\pa x}\right)^2+\left(\frac{\pa w}{\pa y}\right)^2=1.
$$
Functions $s$ and $w$ are related by the following equations
\begin{eqnarray*}
&\ds \na w=\left[
\begin{array}{cc}
p(s)\\
q(s)
\end{array}
\right],\\
&\ds \left[
\begin{array}{cc}
w_{xx} & w_{xy}\\
w_{xy} & w_{yy}
\end{array}
\right]=\frac{d\om(s)}{ds}\,
\left[
\begin{array}{cc}
-w_y\\
w_x
\end{array}
\right]
[s_x \ s_y].
\end{eqnarray*}
(The former is a {\it B\"acklund transformation,} which pairs 
solutions to the Burgers and the eikonal equations mentioned above. It assures 
that $\na s$ and $\na w$ are orthogonal, and the level lines of $s$  are both lines of 
steepest descent and isoclines of  $w.$) The Euclidean metric obeys
$$
dx^2+dy^2=|\na s|^{-2} ds^2+dw^2.
$$
There holds
$$
s(x,y)=t, \quad w(x,y)=\ga(t)
$$
as $x, y$ and $t$ satisfy \eqref{curveapp}.
If $\om$ is free from critical points, the line where
$$
\left[
\begin{array}{c}
x\\
y
\end{array}
\right]=
\left[
\begin{array}{cc}
\al(s)\\
\be(s)
\end{array}
\right]+\frac{\sin\psi(s)}{d\om(s)/ds}\,\left[
\begin{array}{cc}
p(s)\\
q(s)
\end{array}
\right]
$$
is a caustic. (Such a line is an envelope of the level lines of $s.$ Function $s$ develops shocks along it, 
the restriction of $w$ to it equals arc length, and the second-order derivatives of $w$  blow up there.)

\subsection{Borderline initial value problems and caustics.} 
The present item is a recipe for producing caustics, which involves coupling the eikonal equation 
with borderline initial conditions. In the case where refractive index $n$ is identically $1,$ 
any smooth convex curve can be viewed as a caustic provided a GOE is detected, 
whose restriction to the curve in hand equals a relevant arc length.
\par
Let \eqref{curveapp} specify IC and let \eqref{ttimeapp} hold. Let the initial condition imply
\begin{equation*}
\label{icarcapp}
w(x,y)=t
\end{equation*}
as $x, y$ and $t$ satisfy \eqref{curveapp}.
Assume that the appropriate
{\it geodesic curvature} of IC is free from zeros ---
in other words, let \eqref{curveapp} and the following equation
\begin{equation*}
\label{geodapp}
\ka\, (\mbox{velocity})^2=\mbox{Euclidean curvature}-\langle \mbox{unit normal}, \na\log n(x,y) \rangle
\end{equation*}
result in 
\begin{equation*}
\label{kaapp}
\ka\mbox{ \it vanishes nowhere.}
\end{equation*}
\par
Then exactly two GOEs satisfy the present initial condition. 
Both these eikonals fail to exist on both sides of, and be smooth near IC. 
They turn IC into a {\it caustic,} and make the side of it, which
\begin{equation*}
\label{signapp}
(\sgn\,\ka)\times\mbox{(unit normal)}
\end{equation*}
points to, a {\it shadow} region. Either eikonal in hand obeys
\begin{equation*}
\label{wdewapp}
w(x,y)=s\pm\frac{2\sqrt{2}}{3}|r|^{\frac32}|\ka(s)|^{\frac12}+O(r^2), \ \
\De w(x,y)=\pm\frac{1}{\sqrt{2}}|r|^{-\frac12} |\ka(s)|^{\frac12}+O(1)
\end{equation*}
at every point $(x,y)$ that belongs to the light region and is close enough to IC --- in particular, 
the two-sheeted surface, made up of the two eikonals in hand, exhibits an {\it edge of regression} above IC.
\par
Here $r$ and $s$ are the curvilinear coordinates that the following pair
\begin{equation*}
\label{curvcoorapp}
\left[\begin{array}{c}
x\\
y
\end{array}\right]=
\left[\begin{array}{c}
\al(s)\\
\be(s)
\end{array}\right]+
r\bigl(\al'(s)^2+\be'(s)^2\bigr)^{-\frac12}
\left[\begin{array}{c}
-\be'(s)\\
\al'(s)
\end{array}
\right]
\end{equation*}
relates to rectilinear coordinates $x$ and $y$. Coordinate $r$ is a 
signed distance from IC, coordinate $s$ makes $(\al(s(x,y)), \be(s(x,y)))$ 
the orthogonal projection of $(x,y)$ on IC. The former is constant 
on the parallel lines to IC and obeys the following eikonal equation 
$$
\left(\frac{\pa r}{\pa x}\right)^2+\left(\frac{\pa r}{\pa y}\right)^2=1;
$$
the latter is constant on the normal straight-lines to IC and obeys the following Burgers-type equation
$$
\left|\begin{array}{cc}
\pa s/\pa x &\pa s/\pa y\\
\al'(s) &\be'(s)
\end{array}
\right|=0.
$$
Both are subject to the following B\"acklund transformation
$$
\na r=[\al'(s)^2+\be'(s)^2]^{-\frac12}
\left[\begin{array}{c}
-\be'(s)\\
\al'(s)
\end{array}
\right]
$$
and exhibit singularities along the evolute of IC.

\setcounter{subsection}{0}
\setcounter{equation}{0}
\setcounter{theorem}{0}

\section{}

Here we sketch a method of continuing a two-dimensional geometric optical eikonal past a caustic. 
Though rigorous, such a method is slightly reminiscent of 
the so-called theory of complex rays --- cf. \cite{CLOT} or \cite{KFA}, for instance. 
It applies in the case where the refractive index equals $1,$ and can be used for simultaneously 
continuing solutions to non-viscous Burgers equation beyond shock lines. 

\subsection{}  
Our method involves {\it analytic continuation} from the real-number axis into the complex plane --- 
an ill-posed process in the sense of Hadamard. Let $h$ be a real or complex-valued 
function of a real variable, or even a list of samples. 
An analytic continuation of $h$ is a holomorphic function of a complex variable, 
whose domain surrounds the real axis and whose restriction to the real axis fits $h$ well --- in other words, 
a solution $H$ of the following initial value problem for Cauchy-Riemann equation
$$
\frac{\pa H}{\pa y}=i\,\frac{\pa H}{\pa x}, \ \ H(\cdot,0)\simeq h.
$$
\par
If $h$ is an {\it analytic function,} and is {\it not} polluted by noise, 
an analytic continuation $H$ of $h$ results from obvious formulas. For example,
$$
H(x,y)=\sums_{k=0}^\infty \frac{d^k h(x)}{dx^k}\frac{(i y)^k}{k!};
$$
or
$$
H(x,y)=\frac1{2\pi}\ints_{-\infty}^\infty \exp[i \xi (x+iy)]\, \hh(\xi)\, d\xi,
$$
where hat denotes Fourier transformation. 
\par
If $h$ collects {\it gross data,} an effective analytic continuation $H$ of $h$ can be obtained 
by analytically continuing an appropriate, smoothed and denoised version of $h.$ Consider for instance the case where
\begin{eqnarray*}
&-\infty<a<b<\infty,\\
&N=\mbox{ an integer larger than $1,$}
\end{eqnarray*}
a {\it mesh size} is given thus
$$
\De x=(b-a)/(N-1),
$$
{\it nodes} are given by
$$
x_j=a+(j-1) \De x \ \ (j=1,\dots, N);
$$
and $h$ stands for
$$
h_1,\cdots, h_N,
$$
a sequence of {\it noisy samples.}
\par
An ad hoc analytic continuation $H$ solves the following least square problem
$$
\De x\,\sums_{j=1}^N [H(x_j,0)-h_j]^2+\frac{\la}{2L}\,\ints_{a-\De x/2}^{b+\De x/2} dx 
\ints_{-L}^L |H(x,y)|^2 dy=\mbox{minimum}
$$
in a convenient class of holomorphic functions --- e.g. the class
of trigonometric polynomials of a suitable degree. Here
$$
\frac{\la}{2L}\,\ints_{a-\De x/2}^{b+\De x/2} dx 
\ints_{-L}^L |H(x,y)|^2 dy
$$
plays the role of a {\it penalty;} $\la$ and $L$ are {\it regulating parameters} 
--- $\la$ is related to {\it noise,}
$L$ is related to {\it a priori information.}
\par
An explicit expression of $H$ can be derived via discrete Fourier transforms. 
Suppose for simplicity that $N$ is odd, say
$$
N=2n+1 \ \ (n=1,2, \dots);
$$
let 
$$
T=N\,\De x,
$$
and let $DFT$ be the discrete Fourier transform that obeys the following equations
\begin{eqnarray*}
&\ds DFT_k(h)=\sums_{j=1}^N h_j\,\exp\left(-2\pi i k \frac{x_j}{T}\right) \ \ (k=-n, \dots, n),\\
&\ds h_j=\frac1{N}\,\sums_{k=-n}^n DFT_k(h)\,\exp\left(2\pi i k \frac{x_j}{T}\right) \ \ (j=1, \dots, N),\\
&\ds \frac1{N}\,\sums_{k=-n}^n |DFT_k(h)|^2=\sums_{j=1}^N |h_j|^2
\end{eqnarray*}
--- cf. \cite{BH}, for instance. If
$$
C_0=1, \ C_k=\frac{\sinh(4\pi k L/T)}{4\pi k L/T}, \  C_{-k}=C_k \ \ (k=1, \dots, n),
$$
then
$$
H(x,y)=\sums_{k=-n}^{n} (1+\la C_k)^{-1} DFT_k(h)\,\exp\left(2\pi i k \frac{x+iy}{T}\right).
$$
This is a $T$-periodic trigonometric polynomial of degree $n$ that enjoys the 
following properties
\begin{eqnarray*}
&\ds \sums_{j=1}^N |H(x_j,0)-h_j|^2\le \left(\frac{\la}{\la+1/C_n}\right)^2\,\sums_{j=1}^N |h_j|^2,\\
&\ds \frac1{2L}\,\ints_{a-\De x/2}^{b+\De x/2} dx 
\ints_{-L}^L |H(x,y)|^2 dy\le \frac{\De x}{4\la},
\end{eqnarray*}
and
$$
H(x,y)=\sums_{j=1}^N H(x_j,0)\, D_N\left(2\pi \frac{x-x_j+i y}{T}\right)
$$
--- here $D_N$ denotes the Dirichlet, or periodic sinc function obeying
$$
D_N(x)=\frac{\sin(N x/2)}{N \sin(x/2)}
$$
if $x/(2\pi)$ is not an integer.
\par
More information on analytic continuation can be found in \cite{Abd}, \cite{Al}, \cite{BV}, \cite{BK},
\cite{CM}, \cite{CS1}-\cite{CS2}, \cite{DS}, \cite{Dou}-\cite{Dou2}, \cite{Fd1}, \cite{Fd2}, \cite{Fr2}, 
\cite{FITI}, \cite{Gu}, \cite{Ho}, \cite{Lv}, \cite{LA}, \cite{Le}, \cite{LRS}, \cite{Mlr2}-\cite{Mi4}, \cite{MV}, 
\cite{Re}, \cite{Ste}, \cite{Say}, \cite{TA}, \cite{Uz}, \cite{Ve}, \cite{Vu}, \cite{Zha}.

\setcounter{equation}{0}
\setcounter{theorem}{0}

\subsection{} 
Consider a plane curve $C$ that either is inherently smooth or results from 
a suitable smoothing process of raw data. Assume $C$ is {\it analytic} and {\it its curvature vanishes nowhere.} 
For simplicity, assume $C$ is the graph of the following equation
\begin{equation*}
\label{funzapp}
y=f(x),
\end{equation*}
and $f$ is convex.
\par
Alternative parametric representations of $C,$ which are instrumental throughout, include
\begin{equation}
\label{parfunzapp}
x=t, \quad y=f(t),
\end{equation}
where parameter $t$ coincides with the abscissa; and
\begin{equation}
\label{parslopeapp}
x=g'(t), \quad y=t g'(t)-g(t),
\end{equation}
where parameter $t$ is the slope of the tangent straight-line. 
Here $g$ denotes the {\it Legendre conjugate} of $f$ --- recall from e.g. \cite{Rock} that $f$ and $g$ are related thus
\begin{equation*}
\label{legendreapp}
t=f'(x), \ x=g'(t), \ t\, x=f(x)+g(t), \ 1=f''(x)\, g''(t).
\end{equation*}
\par
Curve $C$ changes into a {\it caustic} under the following modus operandi. Let
\begin{equation}
\label{paramapp}
x=\al(t), \quad y=\be(t)
\end{equation}
be any parametric representation of $C,$ where $\al$ and $\be$ are analytic.
Let $\ga$ and $\ka$ stand for arc length and curvature, respectively --- in other words,
\begin{eqnarray*}
\label{}
&\ds\ga(t)=\int\sqrt{\al'(t)^2+\be'(t)^2}\, dt,\\
&\ds\ka=\frac{\al' \be''-\al'' \be'}{[(\al')^2+(\be')^2]^{3/2}}.
\end{eqnarray*}
The following pair
\begin{equation}
\label{curvcoord}
\left[\begin{array}{c}
x\\
y
\end{array}\right]=
\left[\begin{array}{c}
\al(s)\\
\be(s)
\end{array}\right]+
\frac{w-\ga(s)}{\sqrt{\al'(s)^2+\be'(s)^2}}\left[\begin{array}{c}
\al'(s)\\
\be'(s)
\end{array}\right]
\end{equation}
makes $s$ and $w$ curvilinear coordinates. 
As is easy to see,
the level lines of $s$ are {\it tangent straight-lines} to $C,$ 
the level lines of $w$ are {\it involutes} of $C$ --- orthogonal to one another.
\par
Equations \eqref{curvcoord} imply
\begin{equation*}
\label{ineqf}
y\le f(x),
\end{equation*}
and
\begin{equation*}
\label{sw}
s(x,y)=t, \quad w(x,y)=\ga(t),
\end{equation*}
if $x, y$ and $t$ obey \eqref{paramapp} --- 
$s$ and $w$ live below $C$ and satisfy precise conditions along $C.$ 
\par
We compute
\begin{equation*}
\label{jacobiano}
\frac{\pa (x,y)}{\pa(s,w)}=
\left[
\begin{array}{cc}
-\be'(s) &\al'(s)\\
\al'(s) &\be'(s)
\end{array}
\right]
\left[
\begin{array}{cc}
\ka(s)[w-\ga(s)] &0\\
0 &[(\al')^2+(\be')^2]^{-1/2}
\end{array}
\right]
\end{equation*}
to draw the following set:
\begin{equation}
\label{set1}
\frac{\pa}{\pa x}\, \al(s)+\frac{\pa}{\pa y}\, \be(s)=0,
\end{equation}
\begin{equation}
\label{set2}
\left(\frac{\pa w}{\pa x}\right)^2 +\left(\frac{\pa w}{\pa y}\right)^2=1,
\end{equation}
\begin{equation}
\label{set3}
\frac{\pa s}{\pa x} \frac{\pa w}{\pa x}+\frac{\pa s}{\pa y} \frac{\pa w}{\pa y}=0,
\end{equation}
\begin{equation}
\label{set4}
\na w=[\al'(s)^2+\be'(s)^2]^{-1/2}
\left[
\begin{array}{c}
\al'(s)\\
\be'(s)
\end{array}
\right],
\end{equation}
\begin{equation}
\label{set5}
|\na s|^{-2}=[w-\ga(s)]^2 \ka(s)^2 [\al'(s)^2+\be'(s)^2],
\end{equation}
\begin{equation}
\label{set6}
\left[
\begin{array}{cc}
w_{xx} &w_{xy}\\
w_{xy} &w_{y}
\end{array}
\right]=
\frac{1}{\ga(s)-w}\left[
\begin{array}{c}
-w_y\\
w_x
\end{array}
\right] [-w_y\,\, w_x].
\end{equation}
\par
Equation \eqref{set1} is a {\it conservation law;} it reads
\begin{equation*}
\label{burgers}
\frac{\pa s}{\pa x} +s\,\frac{\pa s}{\pa y}=0,
\end{equation*}
the standard {\it Burgers equation,} if  \eqref{parslopeapp} is in force. Equation  \eqref{set2} is the equation of geometrical 
optics in hand. Equation  \eqref{set3} shows that the gradients of $s$ and $w$ are orthogonal. Equation  \eqref{set4}
shows that both $C$ and the tangent straight-lines to $C$ are lines of steepest descent of $w;$ it also shows 
that the straight-lines in question are isoclines of $w.$ Equation \eqref{set4} can be viewed as a {\it B\"acklund 
transformation,} which converts any solution to \eqref{set1}  into a solution to \eqref{set2}. It reads
\begin{equation*}
\label{backlund}
\na w=[1+f'(s)^2]^{-1/2}\left[
\begin{array}{c}
1\\
f'(s)
\end{array}
\right], \quad s=g'\left(\frac{w_y}{w_x}\right),
\end{equation*}
or simply
\begin{equation*}
\label{backlund2}
\na w=[1+s^2]^{-1/2}\left[
\begin{array}{c}
1\\
s
\end{array}
\right], \quad s=\frac{w_y}{w_x},
\end{equation*}
depending on whether \eqref{parfunzapp}  or \eqref{parslopeapp} is in effect. 
Equations \eqref{set5} and \eqref{set6} show that both the gradient 
of $s$ and the second-order derivatives of $w$ blow up near $C.$
\par
We infer that $s$ is governed by a Burgers-type equation, and develops shocks 
along $C.$ The following objects --- $w, C,$ the tangent straight-lines to $C,$ 
and the region below $C$ --- are a {\it geometric optical eikonal,} the relevant {\it caustic,} 
the {\it rays,} and the {\it light region,} respectively.
\par
We now claim: (i) $s$ and $w$ can be {\it continuously extended} into the region where
\begin{equation}
\label{ineqf2}
y>f(x),
\end{equation}
the {\it dark side} of $C,$ if suitable imaginary parts are provided; (ii) the relevant extensions obey 
equations \eqref{set1} to \eqref{set6}.
\par
The points above $C$ are reached by no tangent straight-line to $C,$ of course. We insist in drawing 
tangent straight-lines from these points, but allow complex slopes. In other words, we recast \eqref{curvcoord} 
this way
\begin{equation*}
\label{lines}
\begin{array}{cc}
&-\be'(s) [x-\al(s)]+\al'(s)[y-\be(s)]=0, \\
&[w-\ga(s)][\al'(s)^2+\be'(s)^2]^{1/2}=\al'(s) [x-\al(s)]+\be'(s)[y-\be(s)],
\end{array}
\end{equation*}
and force such equations to hold in the situation where
$$
\re(s)=\la, \quad \im(s)=\mu, \quad \im(x)=\im(y)=0.
$$
\par
The following formulas result
\begin{eqnarray}
\label{analcont}
&\ds\ \ \  x\!=\!\frac{\im[\ovr{\al'(s)}(\al(s)\be'(s)\!-\!\al'(s)\be(s))]}{\im[\ovr{\al'(s)}\be'(s)]}, \ 
y\!=\!\frac{\im[\ovr{\be'(s)}(\be(s)\al'(s)\!-\!\be'(s)\al(s))]}{\im[\ovr{\be'(s)}\al'(s)]},\nonumber \\
\nonumber\\
&s=\la+i\mu,  \\
\nonumber\\
&\ds w=\ga(s)-\frac1{\sqrt{\al'(s)^2+\be'(s)^2}}\left\{\ovr{\al'(s)}\frac{\im\,\be(s)}{\im[\ovr{\al'(s)}\be'(s)]}
+\ovr{\be'(s)}\frac{\im\,\al(s)}{\im[\ovr{\be'(s)}\al'(s)]}\right\},\nonumber
\end{eqnarray}
where $\al, \be$ and $\ga$ stand for the analytic continuations of the original objects.
\par
Formulas \eqref{analcont} answer the claim. Among other things, they give 
\begin{eqnarray*}
\label{asympt}
\ds  \left[\!\begin{array}{c}
x\\
y
\end{array}
\!\right]\!=\!
\left[\!\begin{array}{c}
\al(\la)\\
\be(\la)
\end{array}
\!\right]\!+
\frac{dA(\la)}{d\la}\,
\left[\!\begin{array}{c}
\!\al'(\la)\\
\be'(\la)
\end{array}
\!\right]\,\mu^2+B(\la)\,\left[\!\begin{array}{c}
-\be'(\la)\\
\al'(\la)
\end{array}
\!\right]\,\mu^2+O(\mu^4),\nonumber\\
A=-\frac13\,\log|\ka|-\frac12\log\sqrt{(\al')^2+(\be')^2}, \quad
B=\frac{\ka}{2}\,\sqrt{(\al')^2+(\be')^2},
\end{eqnarray*}
as $\mu$ approaches zero, and
\begin{eqnarray*}
\det\frac{\pa (x,y)}{\pa (\la,\mu)}=
\frac{|\al'(s)\im\,\be(s)-\be'(s)\im\,\al(s)|^2|\al'(s) \be''(s)-\al''(s)\be'(s)|^2}
{\left[\im\,\ovr{\al'(s)}\be'(s)\right]^3};
\end{eqnarray*}
consequently
\begin{equation*}
\label{y-fx}
y-f(x)=\frac12\, f''(\al(\la))\,[\im\, \al(\la+i\mu)]^2+O(\mu^4)
\end{equation*}
and
\begin{equation*}
\label{jacasympt}
\det\frac{\pa (x,y)}{\pa (\la,\mu)}=\mu\, [\al'(\la)\be''(\la)-\al''(\la)\be'(\la)+O(\mu^2)]
\end{equation*}
as $\mu$ approaches zero. Thus \eqref{analcont}  imply \eqref{ineqf2}, as well as 
\begin{equation*}
\label{jacnotzero}
\det\frac{\pa (x,y)}{\pa (\la,\mu)}\not=0
\end{equation*}
if $\mu$ is different from, and sufficiently close to $0.$

\subsection{}
Here is an example.
A {\it catenary} is the graph of either the following equation
\begin{equation*}
\label{caten1}
y =\cosh x
\end{equation*}
or the following equations
\begin{equation}
\label{caten2}
x=\log\bigl(t+\sqrt{1+t^2}\bigr),\quad y=\sqrt{1+t^2},
\end{equation}
where parameter $t$ coincides with both an arc length and the slope of the tangent straight-line.
\par
Consider solutions to the following equations
\begin{equation*}
\label{set23}
\left(\frac{\pa w}{\pa x}\right)^2 +\left(\frac{\pa w}{\pa y}\right)^2=1,\quad
\frac{\pa s}{\pa x} +s\,\frac{\pa s}{\pa y} =0,
\end{equation*}
which obey the following conditions
\begin{equation*}
\label{ws2}
w(x,y)=s(x,y)=t
\end{equation*}
as $x$ and $y$ obey \eqref{caten2}. Function $w$ and the catenary in question are an {\it eikonal} and the relevant {\it caustic,} 
respectively; $s$ obeys {\it Burgers equation,} takes a constant value on each tangent straight-line to the catenary
and develops shocks along the catenary.
\par
The {\it light region} is the set where
\begin{equation*}
-\infty<x<\infty, \quad y\le\cosh x;
\end{equation*}
the {\it shadow region} is the set above the catenary, where
\begin{equation*}
-\infty<x<\infty, \quad \cosh x<y.
\end{equation*}
The following pair
\begin{equation*}
\label{pair1}
\begin{array}{c}
\ds s\, x-y=s\,\log\bigl(s+\sqrt{1+s^2}\bigr)-\sqrt{1+s^2}, \\ 
\ds w=\frac1{\sqrt{1+s^2}}\, \left[x+s\,y-\log\bigl(s+\sqrt{1+s^2}\bigr)\right],
\end{array}
\end{equation*}
makes $s$ and $w$ implicit functions of $x$ and $y$ in the light region. 
Eikonal $w$ and its partner $s$ can be {\it continued} in a subset of the shadow region via the following equations
\begin{eqnarray*}
\label{catcontin}
&\ds x=\la+\tanh\la\, (\mu\cot\mu-1), \ y=\frac1{\cosh\la}\,
\left(\sinh^2\la\,\frac{\mu}{\sin\mu}+\mu\sin\mu+\cos\mu\right),\nonumber\\
&\ds s=\sinh(\la+i\mu),\\
&\ds w=\sinh\la\,\frac{\mu}{\sin\mu}+
\frac{i}{\cosh\la}\,(\sin\mu-\mu\,\cos\mu)\nonumber
\end{eqnarray*}
--- here $\la$ and $\mu$ are parameters such that
$$
-\infty<\la<\infty, \quad 0\le\mu\le\pi/2.
$$
\par
The foregoing equations imply that
\begin{equation*}
\label{asympcat}
\ds y-\cosh x=(\cosh\la)\, \mu^2\,\left[\frac12+O(\mu^2)\right],\\
\end{equation*}
as $\mu$ approaches $0.$ Furthermore,
\begin{eqnarray*}
\label{asympcat2}
&\ds\det\frac{\pa (x,y)}{\pa (\la,\mu)}=
\frac{[\mu^2+\tanh^2\la\,(1-\mu\cot\mu)^2](\sinh^2\la+\cos^2\mu)}{\cosh\la\sin\mu},\nonumber\\
&\ds\frac{\pa w}{\pa x}=\frac{1}{\cosh(\la+i\mu)}, \
\frac{\pa w}{\pa y}=\tanh(\la+i\mu)\nonumber
\end{eqnarray*}
--- in particular, a singularity occurs at the point whose coordinates are
$$
x=0,\quad y=\pi/2.
$$
\par
Figures 5 and 6 show plots of the imaginary parts of $w$ and $s.$

\begin{figure}[tbh]
\label{fig:imagw}
\centerline{
{\epsfxsize=3.5in
\epsfbox{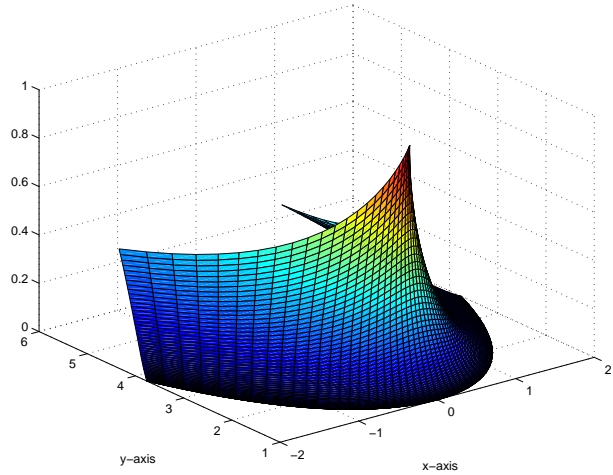}
} 
}
\caption{Eikonal equation: the imaginary part of $w$ beyond a caustic.}
\end{figure}

\begin{figure}[tbh]
\label{fig:imags}
\centerline{
 {\epsfxsize=3.5in
\epsfbox{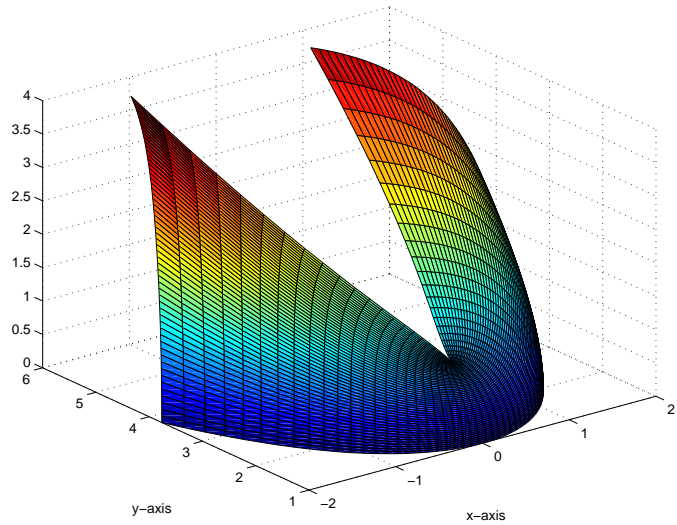}
}
}
\caption{Burgers equation: the imaginary part of $s$ beyond a shock-line.}
\end{figure}

\setcounter{equation}{0}
\setcounter{theorem}{0}

\section{}

\subsection{}  Differentiating a real-valued function of one real variable is among the most 
elementary processes of mathematical and numerical analysis, but is also a significant 
prototype of those problems that are nowadays qualified ill-posed in the sense of Hadamard. 
Methods of approximating derivatives of smooth functions under non-exact data have been 
widely experimented over the years. Here we take the opportunity of sketching one more of such methods. 
We consider the case where data consist of {\it discrete} and {\it noisy} samples, nodes are {\it equally spaced,} 
and information is available on both the relevant noise and the underlying smoothness. 
Our method is inspired by ideas that the theory of statistical learning has recently revived --- see e.g. \cite{CcSm}, \cite{EPP}, 
\cite{SmZh 1}-\cite{SmZh 4}, \cite{Vpn 1}-\cite{Vpn 4}   --- and of course mimics several of its ascendants --- see e.g. \cite{ACR}, \cite{Als}, \cite{AnBl 1}-\cite{AnBl 2}, \cite{AnHg}, \cite{Baa}, \cite{Brt}, \cite{CJW}, \cite{Cox}, \cite{Clm}, \cite{Dvs}, \cite{Dlg}, \cite{DoIv},
\cite{EgKn}, \cite{Fr1}, \cite{Gr}, \cite{JhRs}, \cite{KnMr}, \cite{KnWl}, \cite{Klp}, \cite{LuPr}, \cite{LuWa}, \cite{MlMn}, \cite{Mr 1}-\cite{Mr 2}, \cite{MrGu}, \cite{MMZ}, \cite{Olv}, \cite{Rmm 1}-\cite{Rmm 2}, \cite{RmSm}, \cite{RcRs}, 
\cite{Skl}, \cite{StLn}, \cite{Srv}, \cite{Vsn}, \cite{Wng}.

\par
Items in input include:
\begin{itemize}
\item[(i)] the {\it end points} of a bounded interval --- $a$ and $b;$
\item[(ii)] the {\it number} of both nodes and samples --- an integer $N,$ larger than $2;$
\item[(iii)] {\it nodes} from $a$ to $b$ --- specifically,
$$
x_k=a+(k-1)\,\frac{b-a}{N-1} \quad (k=1,\dots, N);
$$
\item[(iv)] {\it noisy samples} --- a sequence
$$
g_1, g_2,\dots, g_N
$$
of real numbers whatever.
\end{itemize}
\par
Goals include recovering some function $f$  and the derivative $f'$ of $f$ based upon the following information only:
\begin{itemize}
\item[(v)]  $f$ is {\it smooth;}
\item[(vi)] $f(x_k)$ is {\it close} to $g_k,$ for $k = 1, 2,\dots, N.$
\end{itemize}
\par
Our recipe segments into the following three steps. First, let $\la$ and $\mu$ be {\it positive parameters,} 
and solve the following variational problem
\begin{equation}
\label{varpb}
\sums_{k=1}^N [u(x_k)-g_k]^2
+\la\ints_{-\infty}^\infty \left[ \mu^7 (u'''')^2+\mu^{-1} u^2\right]\, dx=\mbox{minimum,}
\end{equation}
under the following condition
\begin{equation}
\label{compfcn}
u \ \mbox{ belongs to Sobolev space } W^{4,2}(-\infty,\infty).	
\end{equation}
Second, adjust $\la$ and $\mu$ properly. Third, take $u,$ $u'$ as approximations of $f,$ $f'.$

\subsection{}
Effective formulas read as follows.
\par
Rudiments of the calculus of variations demonstrate that Problem \eqref{varpb} \& \eqref{compfcn}
possesses a unique solution; moreover that such a solution --- named $u$ throughout --- obeys
\begin{equation}
\label{eqeul}
\la\left(\mu^7\,\frac{d^8 u}{dx^8}+\mu^{-1} u\right)+\sums_{k=1}^N [u(x_k)-g_k]\,\de(x-x_k)=0
\end{equation}
for $-\infty<x<\infty.$ The following features are decisive: equation \eqref{eqeul} is {\it affine;}
$$
\mu^7\,\frac{d^8 }{dx^8}+\mu^{-1} 
$$
is a {\it positive} operator in $L^2(-\infty,\infty),$ whose inverse {\it mollifies}; 
$$
\sums_{k=1}^N [u(x_k)-g_k]\,\de(\cdot-x_k)
$$
is a {\it spike train} or a {\it shah-function.}
\par
Condition \eqref{compfcn} and equation \eqref{eqeul} give
\begin{equation}
\label{soleqeul}
-\la\, u(x)=\sums_{k=1}^N [u(x_k)-g_k]\,K\left(\frac{x-x_k}{\mu}\right)
\end{equation}
for $-\infty<x<\infty.$ Equation \eqref{soleqeul} gives
\begin{equation}
\label{syseqeul}
A\,\left[
\begin{array}{c}
u(x_1)\\\vdots\\u(x_N)
\end{array}\right]+
\la\,\left[
\begin{array}{c}
u(x_1)\\\vdots\\u(x_N)
\end{array}\right]=A\,\left[
\begin{array}{c}
g_1\\\vdots\\g_N
\end{array}\right].
\end{equation}

\begin{figure}[tbh]
\label{fig:K}
\centerline{
{\epsfxsize=3.5in
\epsfbox{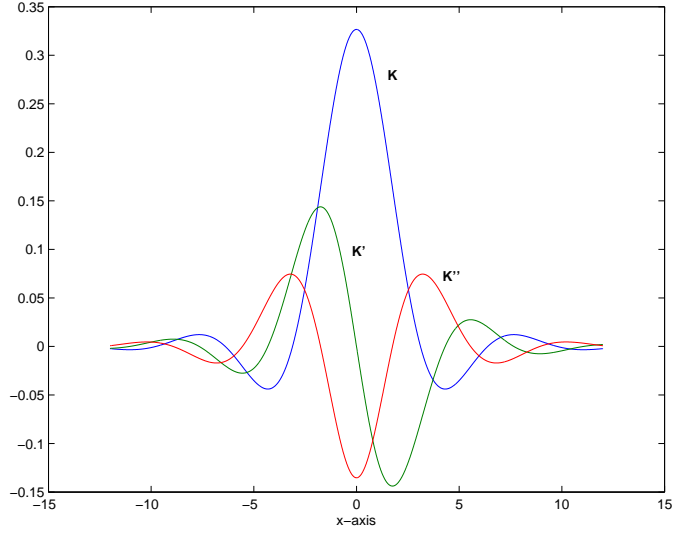}}
}
\caption{Plots of $K, K'$ and $K''.$}
\end{figure}

\par
Here $K$ denotes an appropriate {\it fundamental solution} to $d^8/dx^8 +1,$ namely the solution to
$$
\frac{d^8 K}{dx^8}+K=\de(x) 
$$
that decays at infinity --- $K$ is given by
\begin{equation}
\label{formK}
\pi\, K(x)=\ints_0^\infty\frac{\cos(x\xi)}{1+\xi^8}\, d\xi
\end{equation}					
for $-\infty<x<\infty,$ is {\it even} and {\it positive definite.} Secondly,
\begin{equation*}
\label{formA}
A=\left[
\begin{array}{cccc}
\ds K(0) &\ds K\left(\frac{x_1-x_2}{\mu}\right) & \hdots &\ds K\left(\frac{x_1-x_N}{\mu}\right)\\
\\
\ds K\left(\frac{x_2-x_1}{\mu}\right) & K(0) & \hdots &\ds K\left(\frac{x_2-x_N}{\mu}\right)\\
\vdots & \vdots & \ddots  & \vdots\\
\ds K\left(\frac{x_N-x_1}{\mu}\right) &\ds K\left(\frac{x_N-x_2}{\mu}\right)  & \hdots & K(0)
\end{array}\right]
\end{equation*}	
--- a {\it symmetric, positive definite} Toeplitz matrix.
\par
Let $\nu=\pi/8.$ Manipulating formula \eqref{formK} gives
$$
K(x)=\frac1{8 \sin\nu}-\frac1{8 \cos\nu}\frac{x^2}{2}+\frac1{8 \cos\nu}\frac{x^4}{24}-
\frac1{8 \sin\nu}\frac{x^6}{720}+\frac12 \frac{|x|^7}{5040}+O(x^8)
$$
as $x$ approaches zero --- $K$ behaves near zero like a {\it spline} of order seven. 
Calculus of residues, or convenient formulas from [PBM, Section 2.5.10], give
$$
4 K(x)=e^{-|x| \cos\nu}\cos(|x|\sin\nu-\nu)+e^{-|x| \sin\nu}\sin(|x|\cos\nu+\nu)
$$
as $-\infty<x<\infty.$ Therefore,
\begin{eqnarray*}
\frac{d^n K}{dx^n}(x)=\frac{(-\sgn x)^n}{4}  \left\{e^{-|x| \cos\nu}\cos(|x|\sin\nu-(n+1)\nu)+\right.\\
\left. 
e^{-|x| \sin\nu}\sin(|x|\cos\nu+(n+1)\nu-n\pi/2)\right\}
\end{eqnarray*}
as $x\not=0$ and $n = 1, 2, 3,\cdots.$ Figure 7 shows plots of $K, K', K''.$
\begin{figure}[tbh]
\label{fig:cond}
\centerline{
{\epsfxsize=3.5in
\epsfbox{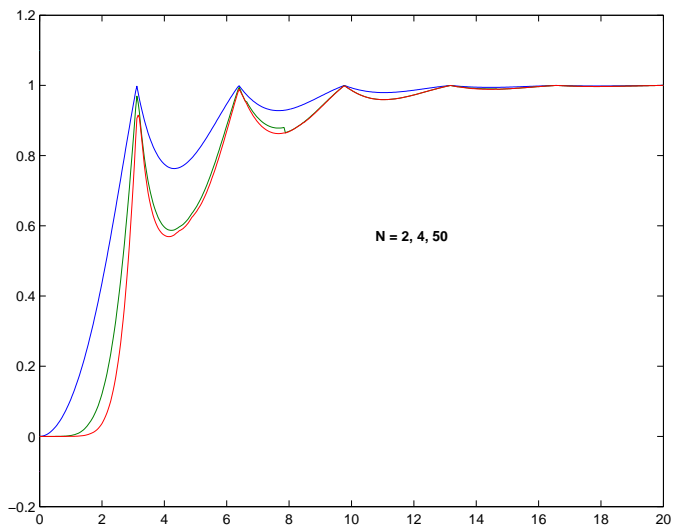}}
}
\caption{Reciprocal condition estimator of $A,$ plotted versus the ratio $(b-a)/((N-1)\mu).$}
\end{figure}
\par
Analysis shows the following. The spectrum of $A$ lies in the open 
interval $]0,N/(8\sin\nu)[.$ 
All eigenvalues of $A$ are close to $1/(8\sin\nu),$ if $\mu (N-1)/(b-a)$ is small; 
otherwise, the largest eigenvalue of $A$ is close to $N/(8\sin\nu)$ and most remaining eigenvalues of $A$ are close to $0.$ 
In particular, $A$ is {\it invertible} anyway; $A$ is either {\it well-conditioned} or {\it ill-conditioned} 
depending on whether $\mu (N-1)/(b-a)$ is small or large. 
Figure 8 shows plots of a reciprocal condition estimator of $A$ versus $(b-a)/((N-1)\mu).$
\par
Equation \eqref{soleqeul} implies that $u$ belongs to the {\it linear span} of
$$
K\left(\frac{\mute-x_1}{\mu}\right),\dots,K\left(\frac{\mute-x_N}{\mu}\right)
$$
--- {\it translations} and {\it dilations} of $K.$ Such items own either a spike-shaped or 
a well-rounded profile depending on whether $\mu$ is small or large, inasmuch as
$$
\frac{\int_{-\infty}^\infty|(d/dx)\mbox{Item}|^2\, dx}{\int_{-\infty}^\infty|\mbox{Item}|^2\, dx}=
\frac57\,\mu^{-2}\tan\nu,\ \
\frac{\int_{-\infty}^\infty|(d^2/dx^2)\mbox{Item}|^2\, dx}{\int_{-\infty}^\infty|\mbox{Item}|^2\, dx}=
\frac37\,\mu^{-4}\tan\nu.
$$
The same items are definitely {\it linearly independent,} although appropriate formulas and analysis show that 
their Gram matrix is well-conditioned only if $\mu (N-1)/(b-a)$ is small enough.
\par
Equation \eqref{syseqeul} determines $u(x_1),\dots, u(x_N)$ in terms of data. 
Note that they
solve
$$
\sums_{k=1}^N [u(x_k)-g_k]^2+\la\,[u(x_1)\,\cdots\, u(x_N)]\, A^{-1}\,\left[
\begin{array}{c}
u(x_1)\\\vdots\\u(x_N)
\end{array}\right]=\mbox{minimum}
$$
--- a standard {\it finite-dimensional least-square} problem, where $\la$ and the inverse of $A$ 
imitate a {\it regulating parameter} \` a la Tikhonov and a {\it penalty,} respectively.
\par
Now we are in a position to draw conclusions. Let $\id$ be the $N\times N$ {\it unit matrix,} and
\begin{equation*}
\label{resolvent}
R=(\la\,\id+A)^{-1}
\end{equation*}
--- a {\it resolvent.} Let $B$ the vector-valued function such that
\begin{equation*}
\label{altbasis}
B(x)=R\,\left[
\begin{array}{c}
\ds K\left(\frac{x-x_1}{\mu}\right)\\\vdots\\\ds K\left(\frac{x-x_N}{\mu}\right)
\end{array}\right]
\end{equation*}
for $-\infty<x<\infty$ --- an {\it alternative basis} in the linear span mentioned above. Define $C$ and $D$ thus
\begin{equation*}
\label{defC}
C=A\,R
\end{equation*}						
\vskip.02cm
\begin{equation*}
\label{defD}
D=\mu^{-1}\,\left[
\begin{array}{cccc}
0 &\ds K'\left(\frac{x_1-x_2}{\mu}\right) & \hdots &\ds K'\left(\frac{x_1-x_N}{\mu}\right)\\
\\
\ds K'\left(\frac{x_2-x_1}{\mu}\right) & 0 & \hdots &\ds K'\left(\frac{x_2-x_N}{\mu}\right)\\
\vdots & \vdots & \ddots  & \vdots\\
\ds K'\left(\frac{x_N-x_1}{\mu}\right) &\ds K'\left(\frac{x_N-x_2}{\mu}\right)  & \hdots & 0
\end{array}\right]\,R.
\end{equation*}			      
The following equations hold.
\begin{equation*}
\label{formu}
u(x)=[g_1\,\cdots\, g_N]\, B(x)
\end{equation*}				          
for $-\infty<x<\infty,$
\begin{equation*}
\label{equ1}
\left[
\begin{array}{c}
u(x_1)\\\vdots\\u(x_N)
\end{array}\right]=
C\,\left[
\begin{array}{c}
g_1\\\vdots\\g_N
\end{array}\right],
\end{equation*}						       
\begin{equation*}
\label{equ2}
\left[
\begin{array}{c}
u'(x_1)\\\vdots\\u'(x_N)
\end{array}\right]=
D\,\left[
\begin{array}{c}
g_1\\\vdots\\g_N
\end{array}\right].
\end{equation*}						       

\begin{figure}[tbh]
\label{fig:6}
\centerline{
{\epsfxsize=3.5in
\epsfbox{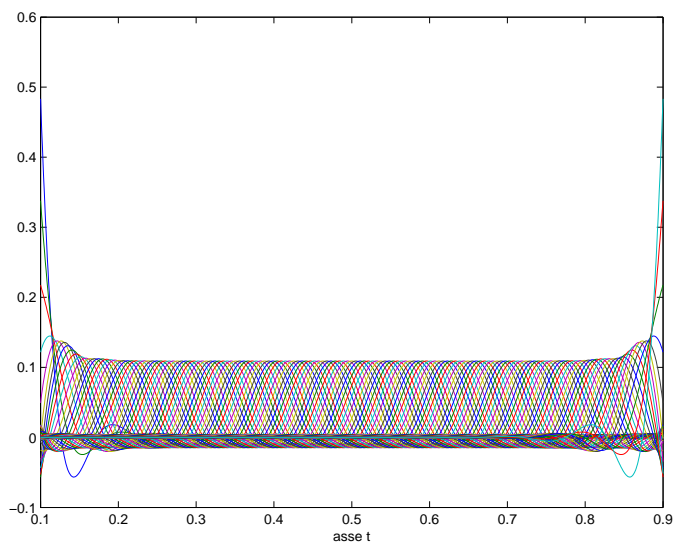}}
}
\caption{The alternative basis $B:$ plots of
$B_1(x),\dots, B_N(x)$ versus $x.$}
\end{figure}

\begin{figure}[tbh]
\label{fig:7}
\centerline{
{\epsfxsize=3.5in
\epsfbox{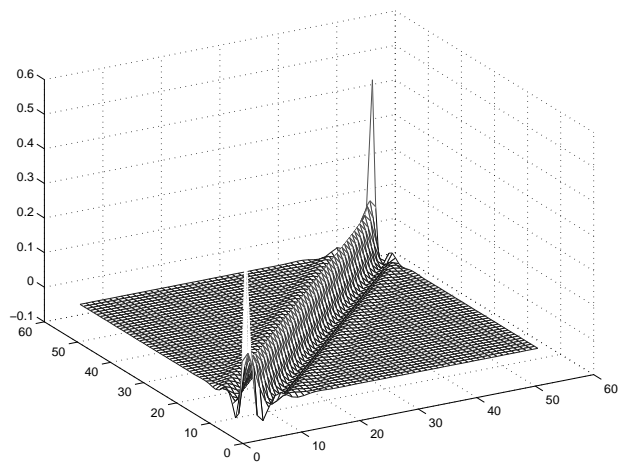}}
}
\caption{Matrix $C$ playing the role of a regularizing filter.}
\end{figure}

\begin{figure}[tbh]
\label{fig:8}
\centerline{
{\epsfxsize=3.5in
\epsfbox{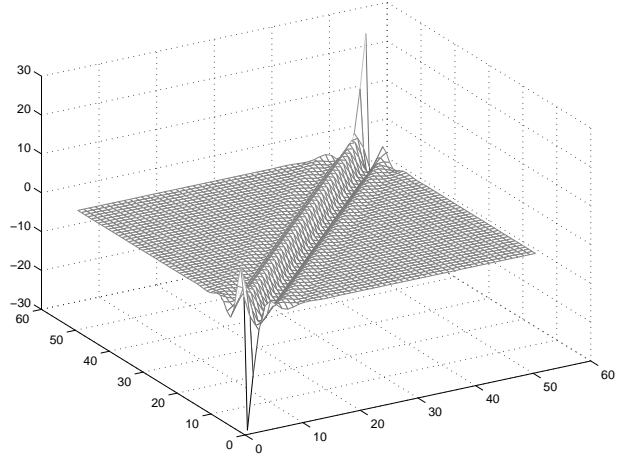}}
}
\caption{Matrix $D,$ simulating differentiation.}
\end{figure}
\par
As figures 9, 10 and 11 show, $B$ and $C$ mimic a typical {\it Green's function} from two different perspectives, 
$D$ mimics a {\it derivative} of a Green's function. Observe incidentally that
$$
\ints_{-\infty}^\infty\left[\mu^7(B'''')^T B''''+\mu^{-1}B^T\,B\right] dx\!=\!
\tr\!\left[A\,(\la\,\id+A)^{-2}\right], \  C=\id-\la\,(\la\,\id+A)^{-1},
$$
and that $B$ and $C$ solve the following variational problem
$$
\la\,\ints_{-\infty}^\infty\left[\mu^7(B'''')^T B''''+\mu^{-1}B^T\,B\right]\, dx+
\tr\left[(C^T-\id)(C-\id)\right]=\mbox{minimum},
$$
subject to the following conditions
$$
B\in [W^{4,2}(-\infty,\infty)]^N, \ \ C=\left[
\begin{array}{cccc}
B_1(x_1) & B_2(x_1) & \hdots & B_N(x_1)\\
B_1(x_2) & B_2(x_2) & \hdots & B_N(x_2)\\
\vdots & \vdots & \ddots  & \vdots\\
B_1(x_N) & B_2(x_N)  & \hdots & B_N(x_N)
\end{array}\right].
$$

\subsection{} Here we offer directions for adjusting $\la$ and $\mu$ properly.
\par
Parameter $\la$ --- {\it dimensionless} --- discriminates whether solution $u$ to Problem \eqref{varpb} \& \eqref{compfcn} 
either fits data well (but is simultaneously sensible of noise), or else is little affected by noise (but departs somewhat from data). 
In loose terms, the following statements hold. First, $u$ virtually {\it interpolates} $g_1,\dots,g_N$  
if $\la$ is close to zero; however, $u$ is liable to own an {\it irregular} profile at the same time. 
Second, $u$ {\it quenches smoothly} if $\la$ grows larger and larger. Indeed,
$$
\ints_{-\infty}^\infty\left[\mu^7(B'''')^T B''''+\mu^{-1}B^T\,B\right]\, dx=\tr A^{-1}+O(\la),
$$
$$
C=\id+O(\la)
$$
as $\la$ approaches zero;
$$
\nr d^nB(x)/dx^n\nr\le\la^{-1}\,\frac{\sqrt{N}}{8\sin[(n+1)\nu]}
$$
as $-\infty<x<\infty$ and $n=0,1,\dots, 6.$
\par
Parameter $\mu$ --- making $\mu/(b-a)$ {\it dimensionless} --- determines how much $u$ is close to, 
or departs from a spike train. Indeed,
$$
B(x)=\mu\,\frac{8\sin\nu}{1+\la\sin\nu}\,\left[
\begin{array}{c}
\de(x-x_1)\\\vdots\\\de(x-x_N)
\end{array}\right]+O(\mu^{3})
$$
as $-\infty<x<\infty$ and $\mu$ approaches zero;
$$
B(x)=\frac1{N+8\la\sin\nu}\,\left[
\begin{array}{c}
1\\\vdots\\1
\end{array}\right]+O(\mu^{-2})
$$
as $-\infty<x<\infty$ and $\mu$ grows larger and larger.
\begin{figure}[tbh]
\label{fig:srq}
\centerline{
{\epsfxsize=3.7in
\epsfbox{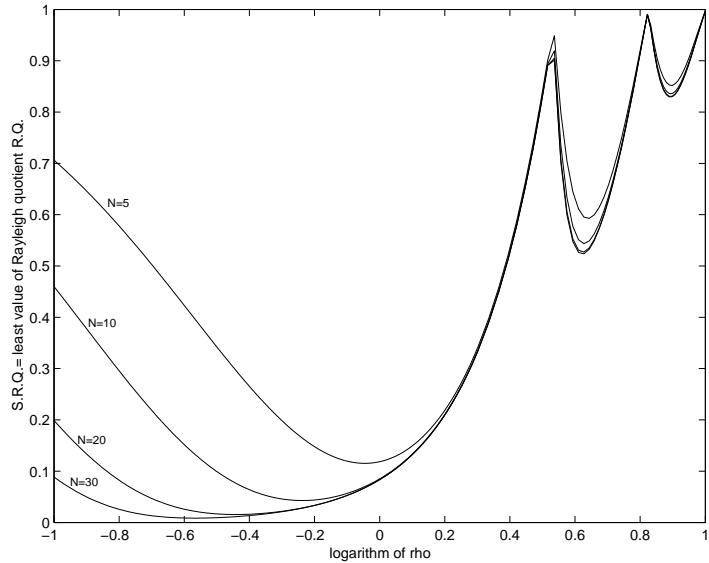}}
}
\caption{Plots of S.R.Q. versus $\log\rho.$ Here $\rho = (b-a)/((N-1)\mu)$.}
\end{figure}
\par
In the case where suitable information is available {\it a priori,} 
a theorem from Subsection \ref{sec:C4} below suggests which values of $\la$ and $\mu$ work properly. 
Otherwise, parameter $\la$ may be determined based upon the {\it discrepancy principle,} 
the {\it cross-validation,} the {\it $L$-curve criterion,} or other customary devices --- see e.g. 
\cite[Chapter 7]{Han} for details.
\par
Parameter $\mu$ is expediently identified by the following recipe. Let
$$
\mbox{R.Q. of } v=\mu^{14}\,\frac{\int_{-\infty}^\infty\left[(d/dx)^7 v(x)\right]^2 dx}{\int_{-\infty}^\infty\left[v(x)\right]^2 dx}
$$
a dimensionless {\it Rayleigh quotient;} define a relevant minimum thus
$$
\mbox{S.R.Q.}=\min\left\{\mbox{R.Q. of } v: 0\not=v\in\,\mbox{span of } K\left(\frac{\mute-x_1}{\mu}\right),\dots, K\left(\frac{\mute-x_N}{\mu}\right)\right\};
$$
then determine $\mu$ so that
\begin{equation}
\label{SRQ}
\mbox{S.R.Q.}=\mbox{minimum}.
\end{equation}
\par
Equation \eqref{SRQ} causes solution $u$ of Problem \eqref{varpb} \& \eqref{compfcn}  to retain 
its most favorable Rayleigh quotient. Equation \eqref{SRQ} can be approached via equation \eqref{SRQ2}
below and tools from linear algebra, although care must be taken of the ill-condition of the involved matrices.
\par
Let
$$
G = \mu^{-1}\times\mbox{Gram matrix of } K\left(\frac{\mute-x_1}{\mu}\right),\dots, K\left(\frac{\mute-x_N}{\mu}\right),
$$
$$
H = \mu^{-1}\times\mbox{Gram matrix of } K^{(7)}\left(\frac{\mute-x_1}{\mu}\right),\dots, K^{(7)}\left(\frac{\mute-x_N}{\mu}\right);    
$$
let $L$ denote the {\it autocorrelation function} of $K,$ videlicet
$$
\pi L(x)=\ints_0^\infty(1+\xi^8)^{-2}\cos(x\xi)\,d\xi
$$
for $-\infty<x<\infty.$ We have
$$
G=\left[L\left(\frac{x_j-x_k}{\mu}\right)\right]_{j,k=1,\dots, N}, \ \ H=\left[-L^{(14)}\left(\frac{x_j-x_k}{\mu}\right)\right]_{j,k=1,\dots, N};
$$
$$
L(x)=\frac78\, K(x)-\frac{x}{8}\,K'(x), \ \ L^{(14)}(x)=\frac78\, K^{(6)}(x)+\frac{x}{8}\,K^{(7)}(x)
$$
for $-\infty<x<\infty,$ and
\begin{equation}
\label{SRQ2}
\mbox{S.R.Q. = least eigenvalue of $H$ with respect to $G$}	
\end{equation}
--- in other words,
$$
\mbox{S.R.Q. = least eigenvalue of $G^{-1/2} H\,G^{-1/2}$}.
$$
\par
Figure 12 shows plots of S.R.Q. versus $(b-a)/((N-1)\mu).$ The following table
\vskip.1cm
$$
\begin{array}{lccc}
N         &     \mbox{minimum value}   &     (b-a)/((N-1)\mu)    &     \mu/(b-a)\\
\\
5	&	0.115320757000	&	0.900127730000	&	0.277738360532\\
10	&	0.042825969700	&	0.578865888000	&	0.191946206219\\
20	&	0.015529762000	&	0.361101444000	&	0.145752889726\\
30	&	0.008487160550	&	0.271108504000	&	0.127191726235\\
40	&	0.005502126710	&	0.220171280000	&	0.116459447577\\
50	&	0.003920874640	&	0.187117585000	&	0.109065982576\\
60	&	0.002967795760	&	0.163520070000	&	0.103651818045\\
70	&	0.002342495540	&	0.145755736000	&	0.099431789245\\
80	&	0.001906845850	&	0.131825897000	&	0.096022315313\\
100	&	0.001349762400	&	0.111466361000	&	0.090619358257\\
120	&	0.001016274330	&	0.097183852500	&	0.086468699567\\
140	&	0.000798662701	&	0.086282601400	&	0.083380015062\\
150	&	0.000716797334	&	0.081805225800	&	0.082041328416
\end{array}
$$
\vskip.1cm\noindent
lists sample solutions to equation \eqref{SRQ}. 
\par
The following formula
\begin{equation}
\label{apprmu}
\frac{\mu}{b-a}
\approx 
0.0415+0.5416\times N^{-1/2}-0.6426\times N^{-1}+1.3706\times N^{-3/2}
\end{equation}
gives an effectual estimate of such solutions.

\subsection{} 
\label{sec:C4}
The following theorem holds.
\begin{theorem}
Suppose $u$ solves problem \eqref{varpb} \& \eqref{compfcn}; suppose $f, \eps$ and $E$ obey
$$
|f(x_k)-g_k|\le\eps \ \ (k=1,\dots, N),
$$
$$
\mu^{-4}\ints_{-\infty}^\infty f(x)^2dx+\mu^{4}\ints_{-\infty}^\infty [f''''(x)]^2dx\le 2(b-a)^{-3} E^2;
$$
let
$$
\de=\max\left\{(1-1/N)^{-1/2}\frac{\eps}{E}, (N-1)^{-2}\right\}.
$$
If $\de$ approaches zero, and
$$
\frac{\la}{N}\,\left(\frac{\mu}{b-a}\right)^3\, \left(\frac{\eps}{E}\right)^{-2}
$$
is bounded and bounded away from zero, then
$$
E^{-1}\max\{ |f(x)-u(x)|: a\le x\le b\}=O(\de^{3/4}), 
$$
$$
E^{-1}(b-a)\,\max\{ |f'(x)-u'(x)|: a\le x\le b\}=O(\de^{1/4}).
$$
\end{theorem}
\begin{proof}
Let
\begin{equation}
\label{defv}
v=f-u.
\end{equation}
\par
Functional $J,$ whose domain is $W^{4,2}(-\infty,\infty)$ and whose value at any trial function $\fhi$ obeys
$$
J(\fhi)=\sums_{k=1}^N [\fhi(x_k)-g_k]^2+\la\,\ints_{-\infty}^\infty \left[\mu^7 (\fhi'''')^2+\mu^{-1} \fhi^2\right]\,dx,
$$
attains its minimum value at $u.$ Consequently,
$$
J(f)=J(u)+(\mbox{a remainder}),
$$
$$
\mbox{remainder }= \sums_{k=1}^N[v(x_k)]^2+\la\,\ints_{-\infty}^\infty \left[\mu^7 (v'''')^2+\mu^{-1} v^2\right]\,dx.
$$
Since
\begin{eqnarray*}
&\ds J(f)\le N \eps^2+2\la\,\left(\frac{\mu}{b-a}\right)^3 E^2,\\
& J(u)\ge 0,\\
& \ints_{-\infty}^\infty \left[\mu^7 (v'''')^2+\mu^{-1} v^2\right]\,dx\ge
2\,\mu^3  \ints_{-\infty}^\infty (v'')^2 dx,
\end{eqnarray*}
we infer
\begin{equation}
\label{est1}
\sums_{k=1}^N[v(x_k)]^2+2\la\mu^3\ints_{-\infty}^\infty (v'')^2 dx\le
E^2\left\{N\,\left(\frac{\eps}{E}\right)^2+2\la\,\left(\frac{\mu}{b-a}\right)^3\right\}.
\end{equation}
\par
Combining inequality \eqref{est1} and Lemma \ref{lm:1} below results in
\begin{eqnarray*}
&&(b-a)^{-1} \ints_a^b v^2 dx\le\\
&&
\quad E^2\left\{\frac{N}{N-1}+\frac1{\pi^4(N-1)^4}\,\frac{N}{2\la}\left(\frac{\mu}{b-a}\right)^{-3}\right\}
\left\{\left(\frac{\eps}{E}\right)^2+\frac{2\la}{N}\,\left(\frac{\mu}{b-a}\right)^3\right\};
\end{eqnarray*}
inequality \eqref{est1} also yields
$$
(b-a)^3 \ints_a^b (v'')^2 dx\le
E^2\,\left\{ 1+\frac{N}{2\la}\,\left(\frac{\mu}{b-a}\right)^{-3}\,\left(\frac{\eps}{E}\right)^2\right\}.
$$
\par
The last two inequalities and a hypothesis imply
\begin{eqnarray}
&(b-a)^{-1} \ints_a^b v^2 dx\le (\mbox{Const.})\, E^2\,\de^2,\label{est2a}\\
&(b-a)^3 \ints_a^b (v'')^2 dx\le (\mbox{Const.})\, E^2.\label{est2b}
\end{eqnarray}
The conclusions follow from \eqref{defv}, \eqref{est2a} and \eqref{est2b}, by virtue of Lemma \ref{lm:2} below.
\end{proof}

\begin{figure}[tbh]
\label{fig:1}
\centerline{
{\epsfxsize=3.5in
\epsfbox{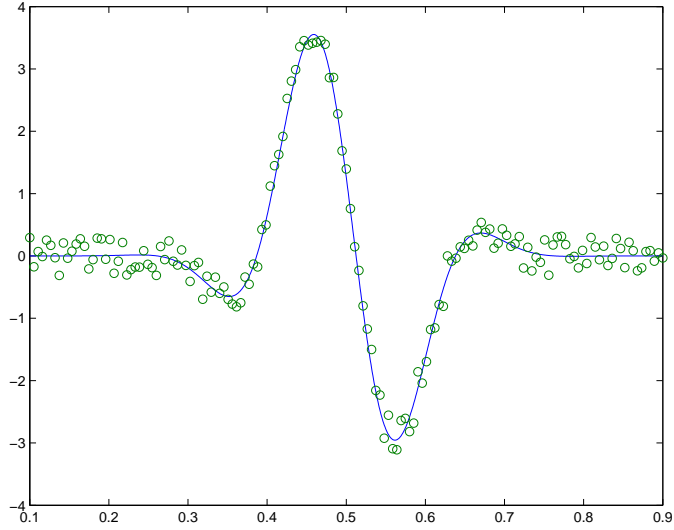}}
}
\caption{Original function (dotted), and samples polluted by noise (circled).}
\end{figure}

\begin{lm}
\label{lm:1}
Let $-\infty<a<b<\infty,$ and $N=2, 3, \dots;$ let
$$
\De x=\frac{b-a}{N-1}, \ \ x_k=a+(k-1)\,\De x \ \ (k=1,\dots, N).
$$
The following inequality
$$
\left\{\ints_a^b v^2 dx\right\}^{1/2}\le (\De x)^{-1/2}\left\{\sums_{k=1}^N v(x_k)^2 dx\right\}^{1/2}+
\left(\frac{\De x}{\pi}\right)^2\left\{\ints_a^b (v'')^2 dx\right\}^{1/2}
$$
holds for every $v$ from $W^{2,2}(a,b).$
\end{lm}
\begin{lm}
\label{lm:2}
Suppose $v$ is in $W^{2,2}(a,b),$ and
$$
(b-a)^{-1/2}\left\{\ints_a^b v^2 dx\right\}^{1/2}=\nr v\nr, \ \ (b-a)^{3}\ints_a^b (v'')^2 dx=1.
$$
The following inequalities hold
\begin{eqnarray*}
&\max|v|\le 2^{1/4}\, 3^{-3/8} \nr v\nr^{3/4}+O(\nr v\nr),\\
&(b-a)\,\max|v'|\le 2^{1/4}\, 3^{-3/8} \nr v\nr^{1/4}+O(\nr v\nr).
\end{eqnarray*}
\end{lm}

\begin{figure}[tbh]
\label{fig:4}
\centerline{
{\epsfxsize=3.5in
\epsfbox{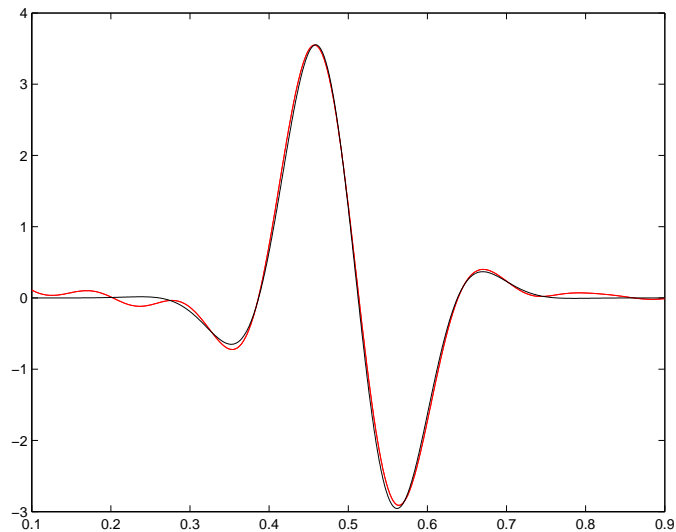}}
}
\caption{Original and recovered functions.}
\end{figure}

\begin{figure}[tbh]
\label{fig:5}
\centerline{
{\epsfxsize=3.5in
\epsfbox{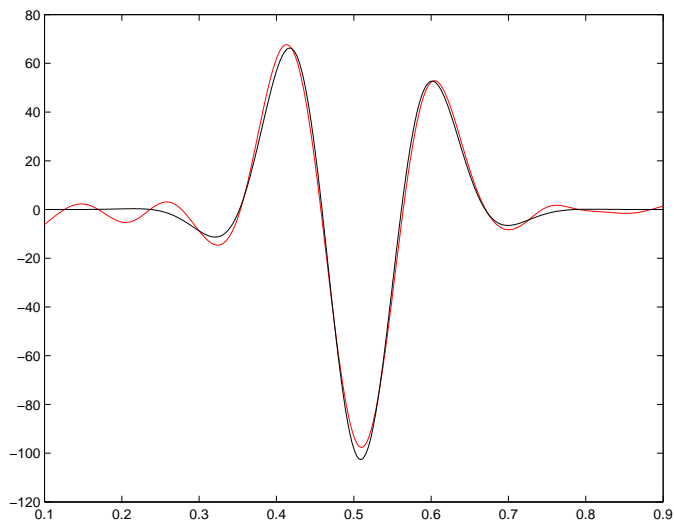}}
}
\caption{Original and recovered derivatives.}
\end{figure}

\par
The proofs of Lemma \ref{lm:1} and \ref{lm:2} are beyond the scope of the present paper, and are omitted.

\subsection{} Here is an example, demonstrating how the present method works.
Let
\begin{eqnarray*}
&a=0, \  b=1,\\
&f(x)=\exp\left[-72\,(x-\frac12)^2)\right][\cos(25 x)-4\,\sin(25 x)],\\
&N=150,\\
&g_k=f(x_k)\pm (\mbox{$5\%$ random noise}) \quad (k=1,\dots, N).
\end{eqnarray*}
Let parameter $\la$ obeys
$$
\la=10^{-6},
$$
let parameter $\mu$ be given by formula \eqref{apprmu},
and let $u$ solve problem \eqref{varpb} \& \eqref{compfcn}. 
Figure 13 plots $g_1, g_2,\dots, g_N$  versus $x_1, x_2,\dots, x_N;$ 
figure 14 plots $f$ and $u,$ and figure 15 plots $f'$ and $u'.$


\begin{thebibliography}{CLOT}



\bibitem[Ab]{Abd} A. Abdukarimov, A problem of analytic continuation from a discrete 
set for functions of complex variables. Approximate solution 
methods and questions of the well-posedness of inverse problems, 5-8, Vychisl. Tsentr, Novosibirsk, 1981.

\bibitem[ACR]{ACR} S. Ahn \& U.J. Choi \& A.G. Ramm, A scheme for stable numerical differentiation. J. Comput. Appl. Math. 186 (2006) 325-334.

\bibitem[Al1]{Als} G. Alessandrini, On differentiation of approximately given functions. Applicable Analysis 11 (1980) 45-59.

\bibitem[Al2]{Al} G. Alessandrini, An extrapolation problem for harmonic functions. Boll. UMI 17B (1980) 860-875.

\bibitem[Ali]{Ali} O.M. Alifanov, Methods of solving ill-posed problems. Translated from the Russian original  in J. Engrg. Phys. 45 (1984) 1237-1245.

\bibitem[AI]{AI} K.A. Ames \& V. Isakov, An explicit stability estimate for an ill-posed Cauchy problem for the wave equation. J. Math. Anal. Appl. 156 (1991) 597-610.

\bibitem[AB1]{AnBl 1} R.S. Anderssen \& P. Bloomfield, Numerical differentiation procedures for non-exact data. Numer. Math. 22 (1973/74) 157-182.

\bibitem[AB2]{AnBl 2} R.S. Anderssen \& P. Bloomfield, A time series approach to numerical differentiation. Technometrics 16 (1974) 69-75.

\bibitem[AH]{AnHg} R.S. Anderssen \& F.R. de Hoog, Finite difference methods for the numerical differentiation of non-exact data. Computing 33 (1984) 259-267.

\bibitem[AT]{AT} D.D. Ang \& D.D. Trong, A Cauchy problem for elliptic equations: quasi-reversibility and error estimates. Vietnam J. Math. 32 (2004), Special Issue, 9-17.

\bibitem[ATY]{ATY}  D.D. Ang \& D.D. Trong \& M. Yamamoto, A Cauchy like problem 
in plane elasticity:regularization by quasi-reversibility with error estimates. 
Vietnam J. Math. 32 (2004), no.2, 197-208.

\bibitem[AK]{AK} M.M. Aripov \& M. Khaidarov, A Cauchy problem for a nonlinear heat equation in an inhomogeneous medium (in Russian). Dokl. Akad. Nauk SSSR 2 (1986) 11-13.

\bibitem[Ba]{Baa} M.I. Baart, Computational experience with the spectral smoothing method for differentiating noisy data. J. Comput. Phys. 42 (1981) 141-151.

\bibitem[BB]{BB} V.M. Babich \& V.S. Buldyrev, Short-wavelength diffraction theory. Springer-Verlag, 1991.

\bibitem[BG]{BG} A.B. Bakushinskii \& A.V. Goncharskii, Ill-posed problems: theory and applications. Kluwer, 1995. 

\bibitem[BKP]{BKP} A.B. Bakushinskii \& M.Yu. Kokurin \& S.K. Paymerov, On error estimates of difference solution method for ill-posed Cauchy problems in a Hilbert space. J. Inverse Ill-Posed Problems 16 (2008) 553-565.

\bibitem[Be]{Bll} J.B. Bell, The noncharacteristic Cauchy problem for a class of equations with time dependence. I : problems in one space dimension. II : multidimensional problems. SIAM J. Math. Anal. 12 (1981), pages 759-777 and 778-797.

\bibitem[BE]{BE} F. Berntsson \& L. Eld\'en, Numerical solution of a Cauchy problem for the Laplace equation. Inverse Problems 17 (2001) 839-854.

\bibitem[Br1]{Brt} M. Bertero, Regularization methods for linear inverse problems. In: Inverse problems (G. Talenti editor), Lecture Notes in Math. vol. 1225, Springer 1986.

\bibitem[Br2]{Brt1} M. Bertero, Linear inverse and ill-posed problems. Advances in Electronics and Electron Physics 75 (1989) 2-120.

\bibitem[Br3]{Brt2} M. Bertero, The use of a priori information in the solution of ill-posed problems. Pages 19-27 in Partial differential equations and applications (P. Marcellini \& G. Talenti \& E. Vesentini editors), Lecture Notes in Pure and Appl. Math. 177, Dekker, 1996.


\bibitem[BDV]{BDV} M. Bertero \& C. DeMol \& G.A. Viano, The stability of inverse problems. Pages 161-214 in Inverse Scattering Problems in Optics (H. Baltes editor), Topics in Current Physics, vol.20, Springer, 1980.

\bibitem[BV1]{BV} M. Bertero \& G.A. Viano, On the numerical analytic continuation  of
the proton electromagnetic form factors. Nuovo Cimento 39 (1965) 1915-1920.

\bibitem[BV]{BV2} M. Bertero \& G.A. Viano, On probabilistic methods for the solution of improperly posed problems. Boll. UMI B (5) 15 (1978) 483-508.

\bibitem[BK]{BK} N. Bogdanova \& T. Kupenova, A numerical method for analytic continuation of holomorphic functions
outside the real axis. Godishnik Vyss. Ucebn. Zaved. Tekh. Fiz. 16 (1979) 91-96.

\bibitem[BM]{BM} D. Bouche \& F. Molinet, M\'ethodes asymptotiques en \'Electromagnétisme. Springer-Verlag, 1994.

\bibitem[Bo1]{Bou1}   L. Bourgeois, A mixed formulation of quasi-reversibility to solve the Cauchy problem for Laplace’s equation. Inverse Problems 21 (2005) 1087-1104.

\bibitem[Bo2]{Bou2}   L. Bourgeois, Convergence rates for the quasi-reversibility method to solve the Cauchy problem for Laplace’s equation. Inverse Problems 22 (2006) 413-430.

\bibitem[BR]{BR} N. Boussetila \& F. Rebbani, Optimal regularization methods for ill-posed problems. Electron. J. Differential Equations 147 (2006) 15 pp.

\bibitem[Bs]{Bre} A. Bressan, An ill-posed Cauchy problem for a hyperbolic system in two space dimensions. Rend. Sem. Mat. Univ. Padova 110 (2003) 103-117.

\bibitem[BH]{BH} W.L. Briggs \& V.E. Henson, The DFT. SIAM, 1995.

\bibitem[Ca]{Cnn} J. Cannon, A Cauchy problem for the heat equation. Ann. Mat. Pura Appl. 66 (1964) 155-165.

\bibitem[CD]{CD} J. Cannon \& J. Douglas, The Cauchy problem for the heat equation. SIAM J. Numer. Anal. 4 (1967) 317-336.

\bibitem[CE]{CE} J. Cannon \& R.E. Ewing, A direct numerical procedure for the Cauchy problem for the heat equation. J. Math. Anal. Appl. 56 (1976) 7-17.

\bibitem[CM]{CM} J.R. Cannon \& K. Miller, Some problems in numerical analytic continuation. SIAM J. Numer. Anal. 2 (1965) 87-98.

\bibitem[Cr]{Crr} A.S. Carasso, A stable marching scheme for an ill-posed initial value problem. Pages 11-35 in Improperly posed problems and their numerical treatment (Oberwolfach, 1982), Internat. Schriftenreihe Numer. Math. 63, Birkhäuser, 1983.

\bibitem[CLOT]{CLOT} S.J. Chapman \& J.M.H. Lawry \& J.R. Ockendon \& R.H. Tew, On the theory of complex rays. SIAM Review 41 (1999) 417-509.

\bibitem[CJW]{CJW} J. Cheng \& X.Z. Jia \& Y.B. Wang, Numerical differentiation and its applications. Inverse Probl. Sci. Eng. 15 (2007) 339-357.

\bibitem[CF1]{CF-1} S. Choudhary \& L.B. Felsen, Asymptotic theory for inhomogeneous waves, IEEE, 
Transactions on Antennas and Propagation 21 (1973) 827-842.

\bibitem[CF2]{CF-2} S. Choudhary \& L.B. Felsen, Analysis of Gaussian beam propagation and diffraction 
by inhomogeneous wave tracking. Proceedings of IEEE 62 (1974) 1530,1541.

\bibitem[CS1]{CS1} S. Ciulli \& T.D. Spearman, Analytic continuation from data points 
with unequal errors. J. Math. Phys. 23 (1982) 1752-1764.

\bibitem[CS2]{CS2} S. Ciulli \& T.D. Spearman, Analytic continuation from empirical data: 
a direct approach to the stabilization problem. J. Math. Phys. 37 (1996) 933-941.

\bibitem[Co]{Co} D. Colton, The noncharacteristic Cauchy problem for parabolic equations in two space variables. Proc. Amer. Math. Soc. 41 (1973) 551-556.

\bibitem[Cx]{Cox} D.D. Cox, Asymptotics of M-type smoothing splines. Ann. Statist. 11 (1983) 530-551.

\bibitem[CS]{CcSm} F. Cucker \& S. Smale, On the mathematical foundations of learning. Bull. Amer. Math. Soc. 39 (2002) 1-49.

\bibitem[Cd]{Cdv} L.A. $\mathrm{\Check{C}udov}$, Difference schemes and ill-posed problems for partial differential equations (in Russian). Pages 34-62 in Computing methods and programming 8, Izdat. Moskow Univ., 1967.

\bibitem[Cl]{Clm} J. Cullum, Numerical differentiation and regularization. SIAM J. Numer. Anal. 8b(1975) 254-265.

\bibitem[Da]{Dvs} A.R. Davies, Optimality in numerical differentiation. Proc. Centre Math. Anal., Austral. Nat. Univ., Canberra, 1988.

\bibitem[De]{De} V.N. Denisov, Stabilization of the solution of the Cauchy problem for the heat equation. Translated from the Russian original in Soviet Math. Dokl. 37 (1988) 688-692.

\bibitem[DS]{DS} Dinh Nho Hao \& H. Sahli, Stable analytic continuation by mollification and 
the fast Fourier transform. Methods of complex and Clifford analysis, 143-152, SAS int. Publ., Delhi, 2004.

\bibitem[Do]{Dlg} T.F. Dolgopolova, Finite-dimensional regularization in numerical differentiation of periodic functions. Ural. Gos. Univ. Mat. Zap. 7 (1969/70) 27-33.

\bibitem[DI]{DoIv} T.F. Dolgopolova \& V.K. Ivanov, On numerical differentiation. USSR Computational Math. and Math. Phys. 6 (1966) 223-232.

\bibitem[DR]{DR} J.R. Dorroh \& X. Ru, The application of the method of quasi-reversibility to the sideways heat equation. J. Math. Anal. Appl. 236 (1999) 503-519.

\bibitem[Dg1]{Dou} J. Douglas, A numerical method for analytic continuation. 
Boundary problems in differential equations, pp.179-189, University of Wisconsin Press, 1960.

\bibitem[Dg2]{Dou2} J. Douglas, Approximate continuation of harmonic and parabolic functions. Pages 353-364 in Numerical solution of partial differential equations (Proc. Symp. Univ. Maryland, 1965), Academic Press, 1966.

\bibitem[Du]{Dui} J.J. Duistermaat, Oscillatory integrals, Lagrange immersions and unfolding of singularities. Comm. Pure Appl. Math. 27 (1974) 207-281.

\bibitem[EK1]{EK1} R.A. Egorchenkov \& Yu.A. Kravtsov, Numerical implementation of complex geometrical optics. Radiophys. and Quantum Electronics 43 (2000), no.7, 569-575 (2001).

\bibitem[EK2]{EK2} R.A. Egorchenkov \& Yu.A. Kravtsov, Diffraction of super-Gaussian beams as described by the complex geometrical optics. Radiophys. and Quantum Electronics 43 (2000), no.10, 798-804 (2001).

\bibitem[EgK]{EgKn} Yu.V. Egorov \& V.A. Kondratiev, On a problem of numerical differentiation. Moscow Univ. Math. Bull. 44 (1989) 85-87.

\bibitem[EF]{EF} P. Einzinger \& L.B. Felsen, Evanescent waves and complex rays. IEEE, Transactions on Antennas and Propagation AP30 (1982) 594-605.

\bibitem[ER]{ER} P. Einzinger \& S. Raz, On the asymptotic theory of inhomogeneous wave tracking. Radio Science 15 (1980) 763-771.

\bibitem[El1]{El1} L. Eld\'en, Approximations for a Cauchy problem for the heat equation. Inverse Problems 3 (1987) 263-273.

\bibitem[El2]{El2} L. Eld\'en, Hyperbolic approximations for a Cauchy problem for the heat equation. Inverse Problems 4 (1988) 59-70.

\bibitem[Ev]{Eva} L.C. Evans, Partial differential equations. Amer. Math. Soc., 1998.

\bibitem[EPP]{EPP} T. Evgeniou \& M. Pontil \& T. Poggio, Regularization networks and support vector machines. Advances in Computational Math. 13 (2000) 1-50.

\bibitem[Ew1]{Ew1} R.E. Ewing, The approximation of certain parabolic equations 
backward in time by Sobolev equations. SIAM  J. Math. Anal. 6 (1975) 91-95.

\bibitem[Ew2]{Ew2} R.E. Ewing, The Cauchy problem for a linear parabolic differential equation. J. Math. Anal. Appl. 71 (1979) 167-186.

\bibitem[EwF]{EwF} R.E. Ewing \& R.S. Falk, Numerical approximation of a Cauchy problem for a parabolic partial differential equation. Math. Comp. 33(1979) 1125-1144.

\bibitem[Fd1]{Fd1} A.M. Fedotov, Theoretical justification of computational algorithms 
for problems of analytic continuation (in Russian). Siberian Math. J. 33 (1992) 511-519.

\bibitem[Fd2]{Fd2} A.M. Fedotov, Analytic continuation of functions 
from discrete sets. J. Inverse Ill-Posed Probl. 2 (1994) 235-252.

\bibitem[Fe1]{Fe-1} L.B. Felsen, Complex-source-point solutions of the field equation and their relation to the propagation and scattering of Gaussian beams. Symposia Mathematica 18 (1976) 39-56.

\bibitem[Fe2]{Fe-2} L.B. Felsen, Evanescent waves, J. Opt. Soc. Am. 66 (1976) 751-760.

\bibitem[Fr1]{Fr1} J.N. Franklin, On Tikhonov method for ill-posed problems. Math. Computation 28 (1974) 889-907.

\bibitem[Fr2]{Fr2} J.N. Franklin, Analytic continuation by the fast Fourier transform. SIAM J. Statistic. Comput. 11 (1990) 112-122.

\bibitem[FITI]{FITI} H. Fujiwara \& H. Imai \& T. Takeuchi \& Y. Iso, Numerical treatment of analytic
continuation with multiple-precision arithmetic. Hokkaido Math. J. 36 (2007) 837-847.

\bibitem[FKN]{FKN} A.A. Fuki \& Yu.A. Kravtsov \& O.N. Naida, Geometrical optics of weakly anisotropic media. Gordon and Breach Science Publishers, 1998.

\bibitem[GZ]{GZ} H. Gajewski \& K. Zacharias, Regularizing a class of ill-posed 
problems for evolution equations (German). J. Math. Anal. Appl. 38 (1972) 784-789.

\bibitem[Go]{Go} J.A. Goldstein, Uniqueness in nonlinear Cauchy problems 
in Banach spaces. Proc. Amer. Math. Soc. 53 (1975) 91-95.

\bibitem[Gr]{Gr} C. W. Groetsch, Differentiation of approximately specified functions.
Amer. Math. Monthly 98 (1991) 847-850.

\bibitem[GS]{GS} V. Guillemin \& S. Sternsberg, Geometric asymptotics. Amer. Math. Soc., 1977.

\bibitem[Gu]{Gu} K.N. Gurjanova, A method of numerical analytic continuation (in Russian). Izv. Vyss. Ucebn. Zaved Matematika 1966 (1966) 47-55.

\bibitem[Ha1]{Had1} J. Hadamard, Sur les problèmes aux d\'eriv\'ees partielles et leur signification physique (in French). Bull. Univ. Princeton 13 (1902) 49-52.

\bibitem[Ha2]{Had2} J. Hadamard, Lectures on Cauchy’s problem in linear partial differential equations. Yale University Press, 1923.

\bibitem[HH]{HH} G. Hammerlin \& K.H. Hoffman (editors), Improperly posed problems and their numerical treatment. Birkhäuser, 1983.

\bibitem[HR]{HR} H. Han \& H.J. Reinhardt, Some stability estimates for Cauchy problems for elliptic equations. J. Inv. Ill-Posed Problems 5 (1997) 437-454.

\bibitem[Hn]{Han} P.C. Hansen, Rank-deficient and discrete ill-posed problems. SIAM 1998.

\bibitem[HF]{HF} E. Heyman \& L.B. Felsen, Evanescent waves and complex rays for 
modal propagation in curved open waveguides. SIAM J. Appl. Math. 43 (1983) 855-884.

\bibitem[He]{Ho} J. van der Hoeven, On effective analytic continuation. Math. Comput. Sci. 1 (2007) 111-175.

\bibitem[Hf]{Hf}B. Hofmann, On the degree of ill-posedness for nonlinear problems. J. Inverse Ill-Posed Problems 2 (1994) 61-76.

\bibitem[Hu]{Hu} Y. Huang, Modified quasi-reversibility method for final 
value problems in Banach spaces. J. Math. Anal. Appl. 340 (2008) 757-769.

\bibitem[HZ]{HZ} Y. Huang \& Q. Zheng, Regularization for ill-posed Cauchy problems 
associated with generators of analytic semigroups. J. Differ. Equations 203 (2004) 38-54.

\bibitem[Is]{Isa} V. Isakov, Inverse problems for partial differential equations. Springer, 1998.

\bibitem[IVT]{IVT} V.K. Ivanov \& V.V. Vasin \& V.P. Tanana, Theory of linear ill-posed problems and its applications (in Russian). Izdat. Nauka, Moskow, 1978.

\bibitem[Jo1]{Jo1} F. John, A note on improper problems in partial differential equations. Comm. Pure Appl. Math. 8 (1955) 494-495.

\bibitem[Jo2]{Jo2} F. John, Numerical solution of the equation of heat conduction for proceeding times. Ann. Mat. Pura Appl. 40 (1955) 129-142.

\bibitem[Jo3]{Jo3} F. John, Continuous dependence on data for solutions with a prescribed bound. Comm. Pure Appl. Math. 13 (1960) 551-585.

\bibitem[JR]{JhRs} L.W. Johnson \& R.D. Riess, An error analysis for numerical differentiation. J. Inst. Math: Appl. 11 (1973) 115-120.

\bibitem[Jn]{Jo} D.S. Jones, The theory of electromagnetism. Pergamon Press, 1964.

\bibitem[Ke]{Ke} J. Keller, Rays, waves and asymptotics. Bull. Amer. Math. Soc. 84 (1978) 727-750.

\bibitem[KL]{KL} J.B. Keller \& R.M. Lewis, Asymptotic methods for partial differential equations: the reduced wave equation and Maxwell¡Çs equations. Plenum Press, 1995.

\bibitem[KO]{KeOk} R. Kenyon \& A. Okounkov, Limit shapes and the complex Burgers equation. Acta Math. 199 (2007) 263-302.

\bibitem[KWYZ]{KWYZ} N. Khanal \& J. Wu \& J.M. Yuan \& B.Y. Zhang, Complex-valued Burgers and KdV-Burgers equations. ArXiv:0901.2132v1 [math.AP] 14 Jan 2009.

\bibitem[KM]{KnMr} J.T. King \& D.A. Murio, Numerical differentiation by finite-dimensional regularization. IMA J. Numer. Anal. 6 (1986) 65-85.

\bibitem[KS]{KS} M.V. Klibanov \& F. Santosa, A computational quasi-reversibility 
method for Cauchy problem for Laplace’s equation. SIAM J. Appl. Math. 51 (1991) 1653-1675.

\bibitem[Kl1]{Kli-1} M. Kline, An asymptotic solution of Maxwell equation. Comm Pure Appl. Math. 4 (1951) 225-262.

\bibitem[Kl2]{Kli-2} M. Kline, Asymptotic solution of linear hyperbolic partial
differential equations. J. Rat. Mech. Anal. 3 (1954) 315-342.

\bibitem[KK]{KK} M. Kline \& J.W. Kay, Electromagnetic theory and geometrical optics. Wiley, 1965.

\bibitem[KV]{KV} P. Knabner \& S. Vessella, Stabilization of ill-posed Cauchy problems for parabolic equations. Ann. Mat. Pura Appl. 149 (1987) 393-409.

\bibitem[Kn]{Kno} R.J. Knops, Instability and the ill-posed Cauchy problem in elasticity. Pages 357-382 in Mechanics of solids, Pergamon Press, 1982.

\bibitem[KW]{KnWl} I. Knowles \& R. Wallace, A variational method for numerical differentiation. Numer. Math. 70 (1995) 91-110.

\bibitem[Ko]{Klp} E.V. Kolpakova, Numerical solution of the problem of reconstructing the derivative
(in Russian). Differencial'nye Uravnenija I Vychisl. Mat. Vyp. 6 (1976) 137-143.

\bibitem[Kr1]{K-1} Yu.A. Kravtsov, A modification of the geometrical optics method (in Russian). Radiofizika 7 (1964) 664-673.

\bibitem[Kr2]{K-2} Yu.A. Kravtsov, Asymptotic solutions of Maxwell¡Çs equations near a caustic (in Russian). Radiofizika 7 (1964) 1049-1056.

\bibitem[Kr3]{K-3} Yu.A. Kravtsov, Complex rays and complex caustics (in Russian). Radiofizika 10 (1967) 1283-1304.

\bibitem[KFA]{KFA} Yu.A. Kravtsov \& G.W: Forbes \& A.A: Asatryan, 
Theory and applications of complex rays. Progress in optics, vol. XXXIX, 1-62, North-Holland, 1999.

\bibitem[KO1]{KO-1} Yu.A. Kravtsov \& Yu.I. Orlov, Geometrical optics of inhomogeneous media. Springer-Verlag, 1990.

\bibitem[KO2]{KO-2} Yu.A. Kravtsov \& Yu.I. Orlov, Caustics, catastrophes and wave fields. Springer-Verlag, 1999.

\bibitem[La]{La} R. Latt\e s, Non-well-set problems and the method of quasi reversibility. Functional Analysis and Optimization pp. 99-113, Academic Press, 1966.

\bibitem[LL]{LL} R. Latt\e s \& J.L. Lions, M\'ethode de quasi-reversibilit\'e et applications
(in French). Dunod, 1967.

\bibitem[Lv1]{La1} M.M. Lavrentiev, On the Cauchy problem for Laplace equation (in Russian). Dokl. Akad. Nauk SSSR 102 (1952).

\bibitem[Lv2]{La2} M.M. Lavrentiev, On the Cauchy problem for Laplace equation (in Russian). Izvest. Akad. Nauk SSSR 120 (1956) 819-842.

\bibitem[Lv3]{La3} M.M. Lavrentiev, On the Cauchy problem for linear elliptic equations of second order (in Russian). Dokl. Akad. Nauk SSSR 112 (1957) 195-197.

\bibitem[Lv4]{La4} M.M. Lavrentiev, Some improperly posed problems of mathematical physics. Springer, 1967.

\bibitem[Lv5]{Lv} M.M. Lavrentiev, Improperly posed problems of Mathematical Physics. Amer. Math. Soc., 1986.

\bibitem[LA]{LA} M.M. Lavrentiev \& B.K. Amonov, Determination of the solution of the diffusion equation from its values on discrete sets (in Russian). Dokl. Akad. Nauk SSSR 221 (1975) 1284-1285.

\bibitem[LRS]{LRS} M.M. Lavrentiev \& V.G. Romanov \& S.P. Shishatskii, Ill-posed problems of mathematical physics and analysis. Amer. Math. Soc., 1986.

\bibitem[LV]{LV} M.M. Lavrentiev \& V.G. Vasiliev, On the formulation of some improperly posed problems of mathematical physics (in Russian). Sibirsk Mat. Ž. 7 (1966) 559-576.

\bibitem[Lw]{Law} J. D. Lawrence, A catalog of special curves. Dover Publications, 1972.

\bibitem[LeV]{LeV} H.A. Levine \& S. Vessella, Estimates and regularization for solutions of some ill-posed problems of elliptic and parabolic type. Rend. Circ. Mat. Palermo 123 (1980) 161-183.

\bibitem[Le]{Le} R.M. Lewis, Analytic continuation using numerical methods. 
Pages 45-81 in: Methods in Computational Physics 4 (B. Adler \& S. Fernbach \& M. Rotenberg editors), 
Academic Press 1965.

\bibitem[LBL]{LBL} R.M. Lewis \& N. Bleistein \& D. Ludwig, Uniform asymptotic theory of creeping waves. Comm. Pure Appl. Math. 20 (1967) 295-328.

\bibitem[Li]{Lns} J.L. Lions, Sur la stabilization de certaines probl\e mes mal pos\'es (in French). Rend. Sem. Mat. Fis. Milano 36 (1966) 80-87.

\bibitem[LP]{LuPr} S. Lu \& S.V. Pereverzev, Numerical differentiation from a viewpoint of regularization theory. Math. Comp. 75 (2006) 1853-1870.

\bibitem[LW]{LuWa} S. Lu \& Y.B. Wang, First and second order numerical differentiation with Tikhonov regularization. Front. Math. China 1 (2006) 354-367.

\bibitem[Lu1]{Lud-1}  D. Ludwig, Uniform asymptotic expansions at a caustic. Comm. Pure Appl. Math. 19 (1966) 215- 250.

\bibitem[Lu2]{Lud-2} D. Ludwig, Uniform asymptotic expansion of the field scattered by a convex object at high frequencies. Comm. Pure Appl. Math. 20 (1967) 103-138.

\bibitem[Ln]{Lun} R.K. Lunenburg, Mathematical theory of optics. Univ. of California Press, 1964.

\bibitem[MT1]{MaTa1} R. Magnanini \& G. Talenti, On complex-valued
solutions to a 2-D eikonal equation. Part One: qualitative properties, 
Contemporary Math. 283 (1999), 203--229.

\bibitem[MT2]{MaTa2} R. Magnanini \& G. Talenti, On Complex-Valued Solutions
to a 2D Eikonal Equation. Part Two: Existence Theorems, SIAM J. Math. Anal.
34 (2002) 805--835.

\bibitem[MT3]{MaTa3} R. Magnanini \& G. Talenti, On Complex-Valued Solutions
to a 2D Eikonal Equation. Part Three: analysis of a B\" acklund transformation, Appl. Anal. 85, no. 1--3 (2006), 249--276.

\bibitem[MT4]{MaTa4} R. Magnanini \& G. Talenti, Approaching a partial differential equation of mixed elliptic-hyperbolic type. Pages 263-276 in Ill-posed and Inverse Problems (S.I. Kabanikin \& V.G. Romanov Editors), VSP, Netherlands (2002).

\bibitem[Ma1]{Msl1} V.P. Maslov, Ill posed Cauchy problems for ideal gas equations and their regularization. Exposé 19 in Équations aux dérivées partielles (Saint Jean de Monts, 1987), École Polytechnique, 1987.

\bibitem[Ma2]{Msl2} V.P. Maslov, Resonance ill-posedness (in Russian). Pages 50-62 in Current problems in applied mathematics and in mathematical physics (Russian), Nauka, Moscow, 1988

\bibitem[MF]{MF} V.P. Maslov \& M.V. Fedoriuk, Semi-classical approximation in quantum mechanics. Reidel Publishing Company, 1981.

\bibitem[MO]{MO} V.P. Maslov \& G.A. Omel'yanov, Geometric asymptotics for nonlinear PDE. Amer. Math. Soc., 2001.

\bibitem[MM]{MlMn} G. Miel \& R. Mooney, On the condition number of Lagrangian numerical differentiation. Appl. Math. Comput. 16 (1985) 241-252.

\bibitem[Mi1]{Mlr1} K. Miller, Three circle theorems in partial differential equations and applications to improperly posed problems. Arch. Rational Mech. Anal. 16 (1964) 126-154.

\bibitem[Mi2]{Mlr3} K. Miller, Least square methods for ill-posed problems with a prescribed bound. SIAM J. Math. Anal. 1 (1970) 52-74.

\bibitem[Mi3]{Mlr2} K. Miller, Stabilized numerical analytic prolongation with poles. SIAM J. Appl. Math. 18 (1970) 346-363.

\bibitem[Mi4]{Mi4} K. Miller, Stabilized numerical methods for location of poles by 
analytic continuation. Pages 9-20 in Studies in Numerical Analysis 2, Numerical Solutions of
Nonlinear Problems, SIAM 1970.

\bibitem[Mi5]{Mi} K. Miller, Stabilized quasi-reversibility and 
other nearly-best-possible methods for non-well-posed problems. 
Pages 161-176 from Lecture Notes in Mathematics, vol. 316, Springer, 1973.

\bibitem[MV]{MV} K. Miller \& G.A. Viano, On the necessity of nearly-best-possible methods for analytic continuation of scattering data. J. Math. Phys. 14 (1973) 1037-1048.

\bibitem[Mo]{Mnk} P. Monk, Error estimates for a numerical method for an ill-posed Cauchy problem for the heat equation. SIAM J. Numer. Anal. 23 (1986) 1155-1172.

\bibitem[Mr]{Mor} V.A. Morozov, Methods for solving incorrectly posed problems. Springer, 1984.

\bibitem[Mu1]{Mr 1} D.A. Murio, Automatic numerical differentiation by discrete mollification. Comput. Math. Appl. 13 (1987) 381-386.

\bibitem[Mu2]{Mr 2} D.A. Murio, The mollification method and the numerical solution of ill-posed problems. John Wiley 1993.

\bibitem[MG]{MrGu} D.A. Murio \& L. Guo, Discrete stability analysis of the mollification method for numerical differentiation, Errata. Comput. Math. Appl. 19 (1990) 15-26, Ibidem 20 (1990) 75.

\bibitem[MMZ]{MMZ} D.A. Murio \& C.E. Mejia \& S. Zhan, Discrete mollification and automatic numerical differentiation. Comput. Math. Appl. 35 (1998) 1-16.

\bibitem[Na]{Nsh} M.Z. Nashed, On nonlinear ill-posed problems I: Classes of operator equations and minimization of functional. Pages 351-373 in Nonlinear analysis and applications (V. Lakshmikantham editor), Dekker,1987.

\bibitem[Nt1]{Ntr1} F. Natterer, The finite element method for ill-posed problems. RAIRO Anal. Numer. 11 (1977) 271-278.

\bibitem[Nt2]{Ntr2} F. Natterer, Numerical treatment of ill-posed problems. Pages 142-167 in Inverse Problems (G. Talenti editor), Lecture Notes in Mathematics 1225, Springer 1986.

\bibitem[Ol]{Olv} J. Oliver, An algorithm for numerical differentiation of a function of one real variable. J. Comput. Appl. Math. 6 (1980) 145-160.

\bibitem[Pa1]{Pay1} L.E. Payne, Bounds in the Cauchy problem for the Laplace equation. Arch. Rational Mech. Anal. 5 (1960) 35-45.

\bibitem[Pa2]{Pay2} L.E. Payne, On some non well-posed problems for partial differential equations. Pages 239-263 in Numerical solution of nonlinear differential equations ( Math. Res. Center Conference, Univ. of Wisconsin), Wiley 1966.

\bibitem[Pa3]{Pay3} L.E. Payne, On a priori bounds in the Cauchy problem for elliptic equations. SIAM J. Math. Anal. 1 (1970) 82-89. 

\bibitem[Pa4]{Pay4} L.E. Payne, Some general remarks on improperly posed problems for partial differential equations. Pages 1-30 in Symposium on non-well-posed problems and logarithmic convexity, Lecture Notes in Math. 316, Springer, 1973. 

\bibitem[Pa5]{Pay5} L.E. Payne, Improperly posed problems in partial differential equations. 
Regional Conference Series in Applied Math. no.22, SIAM, 1975.

\bibitem[Pa6]{Pay6} L.E. Payne, On the stabilization of ill-posed Cauchy problems in nonlinear elasticity. Pages 1-10 in Problems of elastic stability and vibrations (Pittsburgh,1981), Contemp. Math. 4, Amer. Math. Soc., 1981.

\bibitem[Pa7]{Pay7} L.E. Payne, Improved stability estimates for classes of ill-posed Cauchy problems. Applicable Anal. 19 (1985) 63-74.

\bibitem[Pa8]{Pay8} L.E. Payne, On stabilizing ill-posed Cauchy problems for the Navier-Stokes equations. Pages 261-271 in Differential equations with applications to mathematical physics, Math. Sci. Engineering 192, Academic Press, 1993.

\bibitem[PS1]{PaSa1} L.E. Payne \& D. Sather, On some non-well-posed Cauchy problems for quasilinear equations of mixed type. Trans. Amer. Math. Soc. 128 (1967) 135-141.

\bibitem[PS2]{PaSa2} L.E. Payne \& D. Sather, On an initial-boundary value problem for a class of degenerate elliptic operators. Ann. Mat. Pura Appl. 78 (1968) 323-338.

\bibitem[Ph]{Pmh} Pham Minh Hien, A stable marching difference scheme for an ill-posed Cauchy problem for the three-dimensional Laplace equation. Vietnam J. Math. 30 (2002) 79-88.

\bibitem[PS]{PS} P. Pol\'a\v cik \& V. \v Sver\'ak, Zeros of complex caloric functions and singularities of complex viscous Burgers equation. ArXiv:math.AP/0612506v1 18 Dec 2006. 

\bibitem[PBM]{PBM} A.P. Prudnikov \& Yu.A. Brychkov \& O.I. Marichev, Integrals and series, Vol. 1: Elementary Functions. New York, Gordon \& Breach, 1986.

\bibitem[Pc1]{Pc1} C. Pucci, Studio col metodo delle differenze di un problema di Cauchy relativo ad equazioni alle derivate parziali del second’ordine di tipo parabolico (in Italian). Ann. Scuola Norm. Sup. Pisa 7 (1953) 205-215.

\bibitem[Pc2]{Pc2} C. Pucci, Sui problem di Cauchy non ben posti (in Italian). Atti Accad. Naz. Lincei Rend. Cl. Sci. Fis. Mat. Natur. 18 (1955) 473-477.

\bibitem[Pc3]{Pc3} C. Pucci, Discussione del problema di Cauchy per le equazioni di tipo ellittico (in Italian). Ann. Mat. Pura Appl. (4) 46 (1958) 131-154. 

\bibitem[Pc4]{Pc4} C. Pucci, Alcune limitazioni per le soluzioni di equazioni di tipo parabolico (in Italian). Ann. Mat. Pura Appl. 48 (1959) 161-172.

\bibitem[Ra1]{Rmm 1} A.G. Ramm, Numerical differentiation. Izv. Vyss. Ucebn. Zaved. Matematika 11 (1968) 131-134.

\bibitem[Ra2]{Rmm 2}	A.G. Ramm, On stable numerical differentiation. Aust. J. Math. Anal. Appl. 5 (2008).

\bibitem[RS]{RmSm} A.G. Ramm \& A.B. Smirnova, On stable numerical differentiation. Math. Comp. 70 (2001) 1131-1153.

\bibitem[RE]{RE} Z. Ranjbar \& L. Eld\'en, Numerical analysis of an ill-posed Cauchy problem for a convection-diffusion equation. Inverse Probl. Sci. Eng. 15 (2007) 191-211.

\bibitem[Ru]{Ra} J. Rauch, Lectures on geometric optics. Pages 385-466 in 
Hyperbolic equations and Frequency Interactionsn (L. Caffarelli \& W. E Editors). Amer. Math. Soc., 1999.

\bibitem[RHH]{RHH} H.J. Reinhardt \& H. Han \& D.N. Hao, Stability and regularization of a discrete approximation to the Cauchy problem for Laplace equation. SIAM J. Numer. Anal. 36 (1999) 890-905.

\bibitem[RS]{RS} H.J. Reinhardt \& F. Seiffarth, On the approximate solution of ill-posed Cauchy problems for parabolic differential equations. Pages 284-298 in Inverse problems: principles and applications in geophysics, technology, and medicine (Potsdam, 1993), Math. Res. 74, Akademie-Verlag, 1993.

\bibitem[Re]{Re} L. Reichel, Numerical methods for analytic continuation and mesh generation. Constr. Approx. 2 (1986) 23-39.

\bibitem[RR]{RcRs} J. Rice \& M. Rosenblatt, Smoothing splines: regression, derivatives and disconvolution. Ann. Statist. 11 (1983) 141-156.

\bibitem[Ri]{Rch} R.D. Richtmyer, Principles of advanced mathematical physics, vol.1. Springer, 1978.

\bibitem[Ro]{Rock} R.T. Rockafellar, Convex Analysis. Princeton Univ. Press, 1970.

\bibitem[Rm]{Ro} V.G. Romanov, A stability estimate for the solution to the ill-posed Cauchy problem for elasticity equations. J. Inverse Ill-Posed Problems 16 (2008) 615-623.

\bibitem[Sa]{Ste} I. Sabba Stefanescu, On the stable analytic continuation with a condition of uniform boundedness. J. Math. Phys. 27 (1986) 2657-2686.

\bibitem[Sk]{Skl} E.I. Sakalauskas, The Galerkin-Tikhonov method of regularization in the problem of numerical differentiation (in Russian). Zh. Vychisl. Mat. I Mat. Fiz. 11 (1984) 1742-1747.

\bibitem[Sy]{Say}	R. Saylor, Numerical elliptic continuation. SIAM J. Numer. Anal. 4 (1967) 575-581.

\bibitem[Se]{Seg} D. von Seggern, CRC Standard curves and surfaces. CRC Press, 1993.

\bibitem[Sh1]{Sho1}  R.E. Showalter, The final value problem for evolution equations. J. Math. Anal. Appl. 47 (1974) 563-572.

\bibitem[Sh2]{Sho2}   R.E. Showalter, Quasi-reversibility of first and second order parabolic equations. Pages 76-84 in Research Notes in Mathematics, no.1, Pitman, 1975.

\bibitem[SZ1]{SmZh 1} S. Smale \& D-X Zhou, Estimating the approximation error in learning theory. Anal. Appl. (Singap.) 1 (2003) 17-41.

\bibitem[SZ2]{SmZh 2} S. Smale \& D-X Zhou, Shannon sampling and function reconstruction from point values. Bull. Amer. Math. Soc. 41 (2004) 279-305.

\bibitem[SZ3]{SmZh 3} S. Smale \& D-X Zhou, Shannon sampling. II. Connections to learning theory. Appl. Comput. Harmon. Anal. 19 (2005) 285-302.

\bibitem[SZ4]{SmZh 4} S. Smale \& D-X Zhou, Learning theory estimates via integral operators and their approximations. Constr. Approx. 26 (2007), 153-172.

\bibitem[St]{Stn} S. Steinberg, Some unusual ill-posed Cauchy problems and their applications. Pages 17-23 in Improperly posed boundary value problems (Conf. Univ. New Mexico, 1974), Res. Notes in Math. 1, Pitman, 1975.

\bibitem[SL]{StLn} T. Strom \& J.N. Lyness, On numerical differentiation. BIT 15 (1975) 314-322.

\bibitem[Su]{Srv} N.S. Surova, An inverstigation of the problem of reconstructing a derivative by using an optimal regularizing integral operator. Numer. Methods Programming 1 (1977) 30-34.

\bibitem[Ta]{Tl} G. Talenti, Sui problemi mal posti (in Italian). Bollettino U.M.I. 15-A (1978) 1-29.

\bibitem[Ti1]{Tkh1} A.N. Tikhonov, On stability of inverse problems (in Russian). Dokl. Akad. Nauk. SSSR 39 (1944) 195-198.

\bibitem[Ti2]{Tkh2}	A.N. Tikhonov, On the solution of ill-posed problems and the method of regularization (in Russian). Dokl. Akad. Nauk. SSSR 151 (1963) 501-504

\bibitem[Ti3]{Tkh3}	A.N. Tikhonov, On the regularization of ill-posed problems (in Russian). Dokl. Akad. Nauk. SSSR 153 (1963) 49-52.

\bibitem[Ti4]{Tkh4}	A.N. Tikhonov, Improperly posed problems of linear algebra and a stable method for their solution (in Russian). Dokl. Akad. Nauk SSSR 163 (1965) 591-594.

\bibitem[Ti5]{Tkh5}	A.N. Tikhonov, On methods of solving incorrect problems. Translated from the Russian original  in Amer. Math. Soc. Transl. (2) 70 (1968) 222-224. 

\bibitem[TA]{TA} A.N. Tikhonov \& V.Ya. Arsenin, Solution of Ill-Posed Problems. Wiley, 1977.

\bibitem[TGSY]{TGSY} A.N. Tikhonov \& A.V. Goncharsky \& V.V. Stepanov \& A.G. Yagola, Numerical methods for the solution of ill-posed problems. Mathematics and its Applications 328, Kluwer, 1995.

\bibitem[TLY]{TLY} A.N. Tikhonov \& A.S. Leonov \& A.G.  Yagola, Nonlinear ill-posed problems, vol. 1 and 2, Applied Mathematics and Mathematical Computation 14, Chapman \& Hall, 1998.

\bibitem[TT]{TT} D.D. Trong \& N.H. Tuan, Stabilized quasi-reversibility method 
or a class of nonlinear ill-posed problems. Electron. J. Differential Equations 84 (2008) pp.35-55.

\bibitem[Uz]{Uz} M.M. Uzakov, Stability in multidimensional problems of analytic continuation (in Russian).
 Questions of well-posedness and methods for the investigation of inverse problems, 129-141, Vychisl. Tsentr, Novosibirsk, 1986.

\bibitem[Va]{Vsn} V.V. Vasin, The stable evaluation of a derivative in space $C(-\infty,\infty)$ (in Russian). 
USSR Computational Math. and Math. Phys 13 (1973) 16-24.

\bibitem[Vp1]{Vpn 1} V. Vapnik, Structure of statistical learning theory. John Wiley 1996.

\bibitem[Vp2]{Vpn 2} V. Vapnik, Statistical learning theory. John Wiley 1998.

\bibitem[Vp3]{Vpn 3} V. Vapnik, The nature of statistical learning theory. Springer-Verlag 2000.

\bibitem[Vp4]{Vpn 4} V. Vapnik, Estimation of dependence based on empirical data. Springer-Verlag 2006.

\bibitem[Ve]{Ve} S. Vessella, A continuous dependence result in the analytic continuation problem. Forum Math. 11 (1999) 695-703.

\bibitem[VFC]{VFC} N.A. Vrobeva \& L.S. Frank \& L.A. Čudov, Difference methods of solution of an ill-posed Cauchy problem for the three-dimensional Laplace equation (in Russian). Pages 147-155 in Computing methods and programming 11: numerical methods in gas dynamics, Izdat. Moskow Univ., 1968.

\bibitem[Vu]{Vu} Vu Kim Tuan, Stable analytic continuation using hypergeometric summation. Inverse Problems 16 (2000) 75-87.

\bibitem[Wa]{W}	G.G. Walter, An alternative approach to ill-posed problems. J. Integral Equations Appl. 1 (1989) 287-301.

\bibitem[Wn]{Wng} J. Wang, Wavelet approach to numerical differentiation of noisy functions. Commun. Pure Appl. Anal. 6 (2007) 873-897.

\bibitem[Zh]{Zha} B.Ts. Zhamsoev, Estimates of stability in problems of analytic continuation (in Russian). 
Methods for solving inverse problems, 52-64, Vychisl. Tsentr, Novosibirsk, 1983.

\bibitem[Zw]{Z} D. Zwillinger, Handbook of differential equations. Academic Press, 1984.







\end{thebibliography}
\end{document}